\documentclass[a4paper]{amsart}

\usepackage[english]{babel}
\usepackage{amsthm, latexsym, gastex, amssymb,url,listings}
\usepackage[mathcal]{eucal}
\usepackage{graphicx,url}

\copyrightinfo{2013}{Simone Ugolini}

\DeclareMathOperator{\lcm}{lcm}

\DeclareMathOperator{\End}{End}
\DeclareMathOperator{\ord}{ord}

\renewcommand{\phi}[0]{\varphi}
\renewcommand{\theta}[0]{\vartheta}
\renewcommand{\epsilon}[0]{\varepsilon}

\newcommand{\N}{\text{$\mathbf{N}$}}
\newcommand{\Z}{\text{$\mathbf{Z}$}}
\newcommand{\Q}{\text{$\mathbf{Q}$}}

\newcommand{\Pro}{\text{$\mathbf{P}^1$}}
\newcommand{\F}{\text{$\mathbf{F}$}}

\newtheorem{theorem}{Theorem}[section]
\newtheorem{lemma}[theorem]{Lemma}

\theoremstyle{definition}

\newtheorem{definition}[theorem]{Definition}
\newtheorem{example}[theorem]{Example}

\theoremstyle{remark}

\numberwithin{equation}{section}

\begin{document}

\bibliographystyle{amsplain}

\date{}

\title[]
{On the iterations of certain maps $x \mapsto k \cdot(x+x^{-1})$ \\ over finite fields of odd characteristic}

\author{S.~Ugolini}
\email{sugolini@gmail.com} 

\begin{abstract}
In this paper we describe the dynamics of certain rational maps of the form $k \cdot (x+x^{-1})$ over finite fields of odd characteristic. 
\end{abstract}

\maketitle
\section{Introduction}
The dynamics of rational maps over finite fields has drawn the attention of some investigators over the last years. Our first work in this area was \cite{SU2}, where we studied the iterations of the map $\theta(x) = x+x^{-1}$ over finite fields of characteristic two relying upon the relation between $\theta$ with an endomorphism over Koblitz curves. After this first work we attempted at a possible description of the dynamics of $\theta$ over finite fields of odd characteristic. In general it seems  that the behaviour of the map $\theta$ over finite fields of odd characteristic is chaotic. Notwithstanding, there are two remarkable exceptions. In \cite{SU35} we gave a complete description of the dynamics of $\theta$ over finite fields of characteristic three and five, being in characteristic three the map $\theta$ conjugated with the square map and in characteristic five related to an endomorphism of the elliptic curve with equation $y^2 = x^3+x$.

In this paper we address the problem of studying the iterations of certain rational maps which are obtained by a slight modification of the map $\theta$, namely maps of the form $k \cdot (x+x^{-1})$, where $k$ is a non-zero element of a prime field. 

If $p$ is an odd prime and $q$ is a $p$-power, then, for any $k \in \F_p^*$, we can define a map $\theta_k$ over the projective line $\Pro(\F_q) = \F_q \cup \{ \infty \}$ as follows:
\begin{displaymath}
\theta_k : x \mapsto 
\begin{cases}
\infty & \text{if $x \in \{0, \infty\}$,}\\
k \cdot (x + x^{-1}) & \text{otherwise}.
\end{cases}
\end{displaymath}

As in \cite{SU2} and \cite{SU35} it is possible to associate a directed graph $G_{\theta_k}^{q}$ with the map $\theta_k$ over the finite field $\F_q$.  More precisely, we can label each node of $G_{\theta_k}^q$ by an element of $\Pro(\F_q)$ and connect with an arrow the nodes $\alpha$ and $\beta$ if $\beta= \theta_k (\alpha)$. We say that an element $\gamma$ of $G_{\theta_k}^q$ is $\theta_k$-periodic if $\theta_k^l (\gamma) = \gamma$ for some positive integer $l$. Moreover, we notice that an element $\gamma$ which is not $\theta_k$-periodic is pre-periodic, since $\theta_k^s (\gamma)$ is periodic for some positive integer $s$. 

We can notice some properties of the digraph $G_{\theta_k}^q$:
\begin{itemize}
\item the indegree of a node $\beta$ of any $G_{\theta_k}^q$ can be $0$, $1$ or $2$. In fact, if $\beta \in \F_q$, then there exists $\alpha \in \F_q$ such that $\theta_k(\alpha) = \beta$ if and only if there exists a root $\alpha$ in $\F_q$ of the quadratic polynomial $p_k(x) = k x^2- \beta x + k$. In particular, we notice that the indegree of $\beta$ is $1$ exactly for $\beta = \pm 2k$, since the discriminant of $p_k$ is $\beta^2 - 4k^2$;
\item any connected component of $G_{\theta_k}^q$ is formed by a cycle, whose elements can be viewed as roots of reverse binary trees.
\end{itemize} 

Constructing some digraphs $G_{\theta_k}^q$ one can notice that their structure is not particularly symmetric. Nonetheless, they present remarkable symmetries when $k$ falls into one of the following three cases. 

\begin{enumerate}
\item \emph{Case 1:} $k \equiv \pm \frac{1}{2} \pmod{p}$.
\item \emph{Case 2:} $k$ is a root of the polynomial $x^2+\frac{1}{4} \in \F_p[x]$, being $p \equiv 1 \pmod{4}$;
\item \emph{Case 3:} $k$ is a root of the polynomial $x^2 + \frac{1}{2} x + \frac{1}{2} \in \F_p [x]$, being $p \equiv 1, 2,$ or $4 \pmod{7}$.
\end{enumerate}

The just mentioned maps generalize, for different reasons, our previous works \cite{SU2} and \cite{SU35}. More precisely, the map $x \mapsto x+x^{-1}$ over finite fields of characteristic $3$ belongs to the family of maps dealt with in case $1$. As regards the map $x \mapsto x+x^{-1}$ over finite fields of characteristic $5$, we notice that such a map belongs to the family of maps studied in case $2$. Finally, the maps studied in case $3$ appear in the definitions of the endomorphisms of degree $2$ of the elliptic curve of equation $y^2 = x^3-35x+98$ over $\F_p$, with $p \equiv 1, 2, 4 \pmod{7}$. The endomorphism ring of such an elliptic curve is isomorphic to $\Z \left[ \frac{1+ i \sqrt{7}}{2} \right]$. The same holds for the endomorphism ring of the Koblitz curve employed in the study of the iterations of the map $x \mapsto x+x^{-1}$ over finite fields of characteristic $2$. 

\section{Case 1: $k \equiv \pm \dfrac{1}{2} \pmod{p}$}
Let $\F_q$ be a finite field with $q$ elements, where $q=p^n$ for some odd prime $p$ and positive integer $n$. 

The iterations of the map $\theta_{\frac{1}{2}}$ over $\Pro (\F_q)$ can be studied relying upon the consideration that $\theta_{\frac{1}{2}}$ is conjugated to the square map. Indeed, 
\begin{equation}\label{case1_eq_1}
\theta_{\frac{1}{2}}  = \psi \circ s_2 \circ \psi, 
\end{equation}
where $s_2$ and $\psi$ are maps defined on $\Pro (\F_q)$ as follows:
\begin{displaymath}
s_2 (x) = 
\begin{cases}
x^2 & \text{if $x \in \F_q$,} \\
\infty & \text{if $x = \infty$,}
\end{cases}
\quad
\psi (x) =
\begin{cases}
\dfrac{x+1}{x-1} & \text{if $x \in \Pro (\F_q) \backslash \{1, \infty \}$,} \\
1 & \text{if $x = \infty$,}\\
\infty & \text{if $x=1$.}
\end{cases}
\end{displaymath}

As regards the map $\theta_{-\frac{1}{2}}$, the iterations of this latter map can be studied considering that
\begin{equation}\label{case1_eq_2}
\theta_{-\frac{1}{2}}  = \psi \circ s_{-2} \circ \psi, 
\end{equation}
where $s_{-2}$ is the map defined on $\Pro (\F_q)$ as follows:
\begin{displaymath}
s_{-2} (x) = 
\begin{cases}
x^{-2} & \text{if $x \in \F_q^*$,} \\
0 & \text{if $x = \infty$,} \\
\infty & \text{if $x = 0$.}
\end{cases}
\end{displaymath}

We state now a result about $\theta_{\frac{1}{2}}$-periodic and $\theta_{-\frac{1}{2}}$-periodic elements.

\begin{lemma}\label{case1_lem_0}
The following hold.
\begin{itemize}
\item The elements $1$ and $-1$ are $\theta_{\frac{1}{2}}$-periodic and form two cycles of length $1$ each.
\item The elements $1$ and $-1$ are $\theta_{- \frac{1}{2}}$-periodic and form a cycle of length $2$.
\item The element $\infty$ is $\theta_{\frac{1}{2}}$-periodic and $\theta_{-\frac{1}{2}}$-periodic of period $1$.
\item An element $\alpha \in \Pro(\F_{q}) \backslash \{-1, 1, \infty \}$ is $\theta_{\frac{1}{2}}$-periodic (resp. $\theta_{- \frac{1}{2}}$-periodic) of period $l$ if and only if $\psi (\alpha)$ is $s_2$-periodic (resp. $s_{-2}$-periodic) of period $l$. Moreover, the integer $l$ is equal to the multiplicative order $\ord_d (2)$ of $2$ (resp. $\ord_d (-2)$ of $-2$) in $(\Z / d \Z)^*$,  where $d$ is the odd multiplicative order of $\psi (\alpha)$ in $\F_{q}^*$. 
\end{itemize}
\end{lemma}
\begin{proof}
The proof is the same as the proof of Lemma 2.1 in \cite{SU35} once one replaces all occurrences of $3^n$, $k$, $\theta$ with $p^n$, $l$ and $\theta_{\frac{1}{2}}$ (resp. $\theta_{-\frac{1}{2}}$).
\end{proof}

The forthcoming theorems furnish a description of the number and the length of the cycles of the graphs $G_{\theta_{\frac{1}{2}}}^{q}$ and $G_{\theta_{-\frac{1}{2}}}^{q}$ respectively.

\begin{theorem}\label{case1_thm_0}
Let $D = \left\{d_1, \dots, d_m \right\}$ be the set of the distinct odd integers greater than $1$ which divide $q-1$. Denote by $ord_{d_i} (2)$ the multiplicative order of $2$ in $(\Z / d_i \Z)^*$. Consider the set
\begin{displaymath}
L = \left\{ \ord_{d_i} (2) : 1 \leq i \leq m \right\} = \left\{l_1, \dots, l_r \right\}
\end{displaymath}
of cardinality $r$, where  $r \leq m$, and the map
\begin{eqnarray*}
l : D & \to & L\\
d_i & \mapsto & \ord_{d_i} (2).
\end{eqnarray*}
Then:
\begin{itemize}
\item $L \cap \{ 1 \} = \emptyset$;
\item the length of a cycle in $G_{\theta_{\frac{1}{2}}}^{q}$ is a positive integer belonging to $L \cup \{1 \}$;
\item in $G_{\theta_{\frac{1}{2}}}^q$ there are exactly three cycles of length $1$ formed respectively by $1, -1$ and $\infty$;
\item for any $k$ such that $1 \leq k \leq r$ there are
\begin{equation*}
c_k = \dfrac{1}{l_k} \cdot \sum_{d_i \in l^{-1} (l_k)}  \phi(d_i)
\end{equation*}
cycles of length $l_k$ in $G_{\theta_{\frac{1}{2}}}^{q}$, being $\phi$ the Euler's totient function;

\item the number of connected components of $G_{\theta_{\frac{1}{2}}}^{q}$   is 
\begin{equation*}
3 + \displaystyle\sum_{k=1}^r c_k.
\end{equation*}
\end{itemize}
\end{theorem}
\begin{proof}
We notice that $L \cap \{ 1  \} = \emptyset$, since $\ord_{d_i} (2) \geq 2$, being $d_i \geq 3$ for any $i$.  According to Lemma \ref{case1_lem_0} the elements $1, -1$ and $\infty$ are $\theta_{\frac{1}{2}}$-periodic of period $1$. The proof of the remaining statements follows the same lines as the proof of Theorem 2.2 in \cite{SU35}, just replacing all occurrences of $3^n$, $-2$ and $\theta$ by $p^n$, $2$ and $\theta_{\frac{1}{2}}$ respectively. As regards the number of connected components of $G_{\theta_{\frac{1}{2}}}^q$, we notice that there are $\displaystyle\sum_{k=1}^r c_k$ connected components due to the cycles formed by the periodic elements of $\Pro (\F_{q}) \backslash \{-1, 1, \infty \}$ and $3$ more connected components due to the three cycles formed by $1, -1$ and $\infty$.  
\end{proof}

\begin{theorem}\label{case1_thm_0_1}
Let $D = \left\{d_1, \dots, d_m \right\}$ be the set of the distinct odd integers greater than $1$ which divide $q-1$. Denote by $ord_{d_i} (-2)$ the multiplicative order of $-2$ in $(\Z / d_i \Z)^*$. Consider the set
\begin{displaymath}
L = \left\{ \ord_{d_i} (-2) : 1 \leq i \leq m \right\} = \left\{l_1, \dots, l_r \right\}
\end{displaymath}
of cardinality $r$, where  $r \leq m$, and the map
\begin{eqnarray*}
l : D & \to & L\\
d_i & \mapsto & \ord_{d_i} (-2).
\end{eqnarray*}
Then:
\begin{itemize}
\item $L \cap \{2 \} = \emptyset$;
\item the length of a cycle in $G_{\theta_{-\frac{1}{2}}}^{q}$ is a positive integer belonging to $L \cup \{1, 2 \}$;
\item the number of cycles of length $1$ in $G_{\theta_{-\frac{1}{2}}}^{q}$ is 
\begin{displaymath}
\begin{cases}
3, & \text{if $p \equiv 1 \pmod{3}$ or $2 \mid n$ and $p \not =3$};\\
1, & \text{otherwise;}
\end{cases}
\end{displaymath}
\item in $G_{\theta_{-\frac{1}{2}}}^q$ there is exactly $1$ cycle of length $2$ formed by $1$ and $-1$; 
\item for any integer $k$ such that $1 \leq k \leq r$ there are
\begin{equation*}
c_k = \dfrac{1}{l_k} \cdot \sum_{d_i \in l^{-1} (l_k)}  \phi(d_i)
\end{equation*}
cycles of length $l_k$ in $G_{\theta_{-\frac{1}{2}}}^{q}$ formed by elements of $\Pro (\F_q) \backslash \{ -1, 1, \infty \}$, being $\phi$ the Euler's totient function;
\item the number of connected components of $G_{\theta_{-\frac{1}{2}}}^{q}$ is 
\begin{displaymath}
2 + \displaystyle\sum_{k=1}^r c_k.
\end{displaymath}
\end{itemize}
\end{theorem}
\begin{proof}
At first we notice that $2 \not \in L$. In fact, if $\ord_d (-2) =2$ for some odd integer $d > 1$, then $(-2)^2 \equiv 1 \pmod{d}$, which is possible only if $d =3$. Since $\ord_3 (-2) =1$, we conclude that $2 \not \in L$.

As regards the $\theta_{-\frac{1}{2}}$-periodic elements of period $1$, we notice that  $\infty$ is $\theta_{-\frac{1}{2}}$-periodic of period $1$ by definition. Moreover, an element $x \in \F_q^{*}$ is $\theta_{-\frac{1}{2}}$-periodic of period $1$ if and only if $-\frac{1}{2} (x+x^{-1}) = x$, namely if and only if $x^{-1} (3 x^2+1) = 0$. Therefore, $x \in \F_q^*$ is $\theta_{-\frac{1}{2}}$-periodic of period $1$ if and only if $p \not =3$ and $x$ is a square root of $- \frac{1}{3}$. Since $-\frac{1}{3}$ is a square in $\F_q^*$ if and only if $p \equiv 1 \pmod{3}$ or $n$ is even and $p \not = 3$, the thesis follows. Moreover, being $2 \not \in L$, in $G^{q}_{\theta_{-\frac{1}{2}}}$ there is exactly one cycle of length $2$ formed by $1$ and $-1$.

All the other assertions regarding the elements of $\Pro (\F_q) \backslash \{-1, 1, \infty \}$ can be proved as in \cite{SU35}, Theorem 2.2, just replacing all occurrences of $3^n$ and $\theta$ by $p^n$ and $\theta_{-\frac{1}{2}}$ respectively. The number of connected components of $G_{\theta_{-\frac{1}{2}}}^{q}$ is equal to the sum of the number of cycles formed by the periodic elements of $\Pro (\F_q) \backslash \{ -1, 1, \infty \}$, the cycle formed by $1$ and $-1$ and the cycle formed by $\infty$.
\end{proof}

The following lemma generalizes Lemma 2.3 of \cite{SU35} and its proof is the same, after the replacement of any occurrence of $3^n$, $s$ and $\theta$ by $p^n$, $s_{-2}$ and $\theta_{-\frac{1}{2}}$ respectively.

\begin{lemma}\label{case1_lem_1}
Let $2^e$, for some positive integer $e$, be the greatest power of $2$ dividing $q-1$.
Let $\gamma \in \F_{q}$ be a non-$\theta_{-\frac{1}{2}}$-periodic element (in particular $\gamma \not \in \{1, -1 \}$). Then, $\theta_{-\frac{1}{2}} (x) = \gamma$ for exactly two distinct elements $x \in \F_{q}$, provided that $\ord(\psi (\gamma)) \not \equiv 0 \pmod{2^e}$, where $\ord(\psi (\gamma))$ is the multiplicative order of $\psi (\gamma)$ in $\F_{q}^*$. If, on the contrary, $ \ord ( \psi (\gamma)) \equiv 0 \pmod {2^e}$, then there is no $x \in \F_{q}$ such that $\theta_{-\frac{1}{2}} (x) = \gamma$.
\end{lemma}

The adapted version of Lemma \ref{case1_lem_1} for the map $\theta_{\frac{1}{2}}$ reads as follows.

\begin{lemma}\label{case1_lem_1_1}
Let $2^e$, for some positive integer $e$, be the greatest power of $2$ dividing $q-1$.
Let $\gamma \in \F_{q}$ be a non-$\theta_{\frac{1}{2}}$-periodic element (in particular $\gamma \not \in \{1, -1 \}$). Then, $\theta_{\frac{1}{2}} (x) = \gamma$ for exactly two distinct elements $x \in \F_{q}$, provided that $\ord(\psi (\gamma)) \not \equiv 0 \pmod{2^e}$, where $\ord(\psi (\gamma))$ is the multiplicative order of $\psi (\gamma)$ in $\F_{q}^*$. If, on the contrary, $ \ord ( \psi (\gamma)) \equiv 0 \pmod {2^e}$, then there is no $x \in \F_{q}$ such that $\theta_{\frac{1}{2}} (x) = \gamma$.
\end{lemma}
\begin{proof}
We observe that $\theta_{\frac{1}{2}} (-1) = -1$, $\theta_{\frac{1}{2}} (1) = 1$ and $\theta_{\frac{1}{2}} (0) = \infty$. Therefore, if $\theta_{\frac{1}{2}} (x) = \gamma$ as in the hypotheses, then $x \not \in \{-1, 0, 1 \}$. The remaining assertions can be proved as in the proof of Lemma 2.3 of \cite{SU35}, replacing any occurrence of $3^n$, $s$, $-2$ and $\theta$ by $p^n$, $s_{-2}$, $2$ and $\theta_{\frac{1}{2}}$ respectively.
\end{proof}

The following theorem is a generalization of Theorem 2.4 of \cite{SU35}.

\begin{theorem}\label{case1_thm_3}
Let $\alpha \in  \Pro (\F_{q})$ be a $\theta_{\frac{1}{2}}$-periodic (resp. $\theta_{-\frac{1}{2}}$-periodic) element. If $\alpha \in \{-1, 1\}$, then $\alpha$ is not root of any tree in $G_{\theta_{\frac{1}{2}}}^q$ (resp. $G_{\theta_{-\frac{1}{2}}}^q$). If $\alpha \not \in \{1, -1 \}$, then $\alpha$ is the root of a reversed binary tree of depth $e$ in $G_{\theta_{\frac{1}{2}}}^q$ (resp. $G_{\theta_{-\frac{1}{2}}}^q$), where $2^e$ is the greatest power of $2$ which divides $q-1$. In particular:
\begin{itemize}
\item there are $2^{k-1}$ vertices at any level $1 \leq k \leq e$ of the tree;
\item the root has one child and all the other vertices at any level $0 < k < e$ have two children;
\item if $\beta \in \F_q$ belongs to the level $k > 0$ of the tree rooted at $\alpha$, then $2^k$ is the greatest power of $2$ dividing $\ord(\psi(\beta))$. 
\end{itemize}
\end{theorem}

\begin{proof}
Firstly we notice that no tree grows on $1$ and $-1$, since
\begin{displaymath}
\begin{array}{ccc}
\theta_{\pm \frac{1}{2}} (x) =  1 & \Leftrightarrow x^2 \mp 2x+1= 0 & \Leftrightarrow (x \mp 1)^2 = 0;\\
\theta_{\pm \frac{1}{2}} (x) =  - 1 & \Leftrightarrow x^2 \pm 2x+1= 0 & \Leftrightarrow (x \pm 1)^2 = 0.
\end{array}
\end{displaymath}  

All other assertions can be proved as in the proof of Theorem 2.4 of \cite{SU35} replacing all occurrences of $3^n$, $\theta$ and $s$ by $p^n$, $\theta_{\frac{1}{2}}$ (resp. $\theta_{-\frac{1}{2}}$) and $s_2$ (resp. $s_{-2}$) respectively.
\end{proof}

\begin{example}
Hereafter the graph $G_{\theta_{-\frac{1}{2}}}^{29}$ is represented. The vertex numbering refers to the exponents $k$ of the powers $\alpha^k$, where $\alpha$ is the root of the Conway polynomial $x-2 \in \F_{29} [x]$.
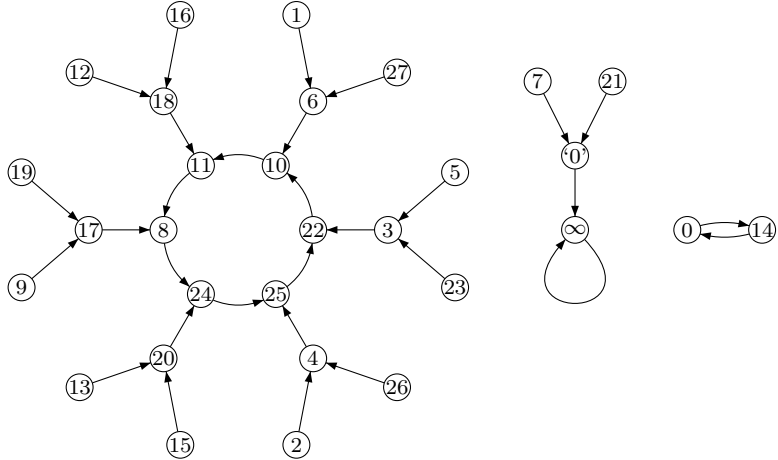
\begin{center}
\begin{picture}(90, 65)(-25,-35)
	\unitlength=2.8pt
    \gasset{Nw=3.6,Nh=3.6,Nmr=1.8,curvedepth=-1.5}
    \thinlines
    \footnotesize
    \node(N1)(10,0){$22$}
    \node(N2)(5,8.66){$10$}
    \node(N3)(-5,8.66){$11$}
    \node(N4)(-10,0){$8$}
    \node(N5)(-5,-8.66){$24$}
    \node(N6)(5,-8.66){$25$}
    
    \node(N11)(20,0){$3$}
    \node(N12)(10,17.32){$6$}
    \node(N13)(-10,17.32){$18$}
    \node(N14)(-20,0){$17$}
    \node(N15)(-10,-17.32){$20$}
    \node(N16)(10,-17.32){$4$}

 	\node(N21)(28.97,7.76){$5$}
    \node(N22)(21.21,21.21){$27$}
    \node(N23)(7.76,28.97){$1$}
    \node(N24)(-7.76,28.97){$16$}
    \node(N25)(-21.21,21.21){$12$}
    \node(N26)(-28.97,7.76){$19$}
    \node(N27)(-28.97,-7.76){$9$}
    \node(N28)(-21.21,-21.21){$13$}
    \node(N29)(-7.76,-28.97){$15$}
    \node(N30)(7.76,-28.97){$2$}
    \node(N31)(21.21,-21.21){$26$}
    \node(N32)(28.97,-7.76){$23$}

    \drawedge(N1,N2){}
    \drawedge(N2,N3){}
    \drawedge(N3,N4){}
    \drawedge(N4,N5){}
    \drawedge(N5,N6){}
    \drawedge(N6,N1){}

    \gasset{curvedepth=0}
     
    \drawedge(N11,N1){}
    \drawedge(N12,N2){}
    \drawedge(N13,N3){}
    \drawedge(N14,N4){}
    \drawedge(N15,N5){}
    \drawedge(N16,N6){}
    
    \drawedge(N21,N11){}
    \drawedge(N22,N12){}
    \drawedge(N23,N12){}
    \drawedge(N24,N13){}
    \drawedge(N25,N13){}
    \drawedge(N26,N14){}
    \drawedge(N27,N14){}
    \drawedge(N28,N15){}
    \drawedge(N29,N15){}
    \drawedge(N30,N16){}
    \drawedge(N31,N16){}
    \drawedge(N32,N11){}
    
    \node(NN1)(45,0){$\infty$}
    \node(NN2)(45,10){`0'}
    \node(NN3)(40,20){$7$}
    \node(NN5)(50,20){$21$}
    
    \node(NN4)(60,0){$0$}
    \node(NN6)(70,0){$14$}
   
    \drawedge(NN5,NN2){}
    \drawedge(NN2,NN1){}
 
    \drawedge(NN3,NN2){}
    \gasset{curvedepth=1}
       \drawedge(NN4,NN6){}
        \drawedge(NN6,NN4){}
    \drawloop[loopangle=-90](NN1){}
\end{picture}
\end{center}

We notice that $q=29$ and $q-1 = 2^2 \cdot 7$. According to the notation of Theorem \ref{case1_thm_0}, $D = \{ 7 \}$ and $\ord_7 (-2) = 6$. This implies that there exists $\dfrac{1}{6} \cdot \phi(7) = 1$ cycle of length $6$. According to the same theorem the elements $1$ and $-1$ (which here are denoted by $0$ and $14$) form a cycle of length $2$, while $\infty$ forms a cycle of length $1$. Moreover, being $e = 2$ according to the notation of Theorem \ref{case1_thm_3}, any cyclic element different from $1$ and $-1$ is root of a binary tree of depth $2$.

\end{example}

\begin{example}
In this example the graph of $G_{\theta_{-\frac{1}{2}}}^{7^2}$ is represented. The vertex numbering refers to the exponents $k$ of the powers $\alpha^k$, where $\alpha$ is a root of the Conway polynomial $x^2 - 2 x + 3 \in \F_7 [x]$.
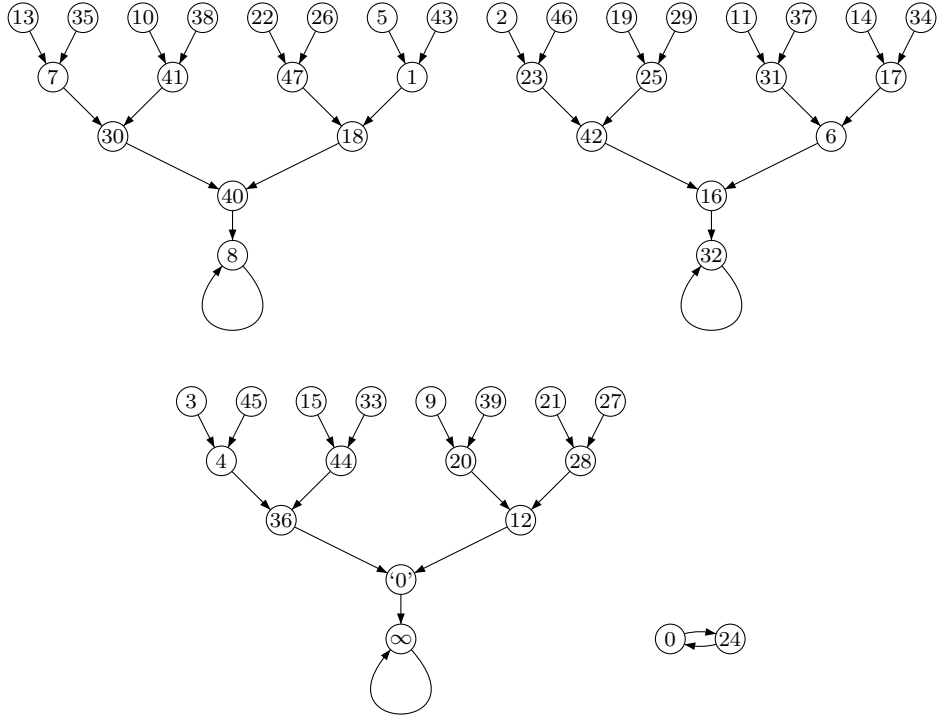
 \begin{center}
    \unitlength=2.8pt
    \begin{picture}(125, 45)(5,-0)
    \gasset{Nw=4,Nh=4,Nmr=3,curvedepth=0}
    \thinlines
   \footnotesize
    
    \node(B1)(36,8){8}
    
    \node(C1)(36,16){40}
    
    \node(D1)(20,24){30}
    \node(D2)(52,24){18}
              
    \node(E1)(12,32){7}
   	\node(E2)(28,32){41}
    \node(E3)(44,32){47}
    \node(E4)(60,32){1}

	\node(F1)(8,40){13}
    \node(F2)(16,40){35}
    \node(F3)(24,40){10}          
    \node(F4)(32,40){38}
    \node(F5)(40,40){22}
    \node(F6)(48,40){26}
    \node(F7)(56,40){5}
    \node(F8)(64,40){43}
    
    \node(BB1)(100,8){32}
    
    \node(CC1)(100,16){16}
    
    \node(DD1)(84,24){42}
    \node(DD2)(116,24){6}
              
    \node(EE1)(76,32){23}
   	\node(EE2)(92,32){25}
    \node(EE3)(108,32){31}
    \node(EE4)(124,32){17}

	\node(FF1)(72,40){2}
    \node(FF2)(80,40){46}
    \node(FF3)(88,40){19}          
    \node(FF4)(96,40){29}
    \node(FF5)(104,40){11}
    \node(FF6)(112,40){37}
    \node(FF7)(120,40){14}
    \node(FF8)(128,40){34}

	\drawedge(F1,E1){}	
    \drawedge(F2,E1){}
    \drawedge(F3,E2){}
    \drawedge(F4,E2){}
    \drawedge(F5,E3){}
    \drawedge(F6,E3){}
    \drawedge(F7,E4){}
    \drawedge(F8,E4){}
    \drawedge(E1,D1){}
    \drawedge(E2,D1){}
    \drawedge(E3,D2){}
    \drawedge(E4,D2){}
    \drawedge(D1,C1){}
    \drawedge(D2,C1){}
    \drawedge(C1,B1){}
    \drawloop[loopangle=-90](B1){} 
    
    	\drawedge(FF1,EE1){}	
    \drawedge(FF2,EE1){}
    \drawedge(FF3,EE2){}
    \drawedge(FF4,EE2){}
    \drawedge(FF5,EE3){}
    \drawedge(FF6,EE3){}
    \drawedge(FF7,EE4){}
    \drawedge(FF8,EE4){}
    \drawedge(EE1,DD1){}
    \drawedge(EE2,DD1){}
    \drawedge(EE3,DD2){}
    \drawedge(EE4,DD2){}
    \drawedge(DD1,CC1){}
    \drawedge(DD2,CC1){}
    \drawedge(CC1,BB1){}
    \drawloop[loopangle=-90](BB1){} 
    \end{picture}
  \end{center}
  
   \begin{center}
    \unitlength=2.8pt
    \begin{picture}(80, 55)(5,-5)
    \gasset{Nw=4,Nh=4,Nmr=3,curvedepth=0}
    \thinlines
   \footnotesize
    
    \node(B1)(36,8){$\infty$}
    
    \node(C1)(36,16){`0'}
    
    \node(D1)(20,24){36}
    \node(D2)(52,24){12}
              
    \node(E1)(12,32){4}
   	\node(E2)(28,32){44}
    \node(E3)(44,32){20}
    \node(E4)(60,32){28}

	\node(F1)(8,40){3}
    \node(F2)(16,40){45}
    \node(F3)(24,40){15}          
    \node(F4)(32,40){33}
    \node(F5)(40,40){9}
    \node(F6)(48,40){39}
    \node(F7)(56,40){21}
    \node(F8)(64,40){27}

	\drawedge(F1,E1){}	
    \drawedge(F2,E1){}
    \drawedge(F3,E2){}
    \drawedge(F4,E2){}
    \drawedge(F5,E3){}
    \drawedge(F6,E3){}
    \drawedge(F7,E4){}
    \drawedge(F8,E4){}
    \drawedge(E1,D1){}
    \drawedge(E2,D1){}
    \drawedge(E3,D2){}
    \drawedge(E4,D2){}
    \drawedge(D1,C1){}
    \drawedge(D2,C1){}
    \drawedge(C1,B1){}
    \drawloop[loopangle=-90](B1){} 
    
    \node(BB1)(72,8){0}
    
    \node(CC1)(80,8){24}
    
    \gasset{curvedepth=1.0}
    \drawedge(BB1,CC1){}
    \drawedge(CC1,BB1){}
    \end{picture}
  \end{center}
We note that $q=7^2$ and $q-1 = 2^4 \cdot 3$. According to the notation of Theorem \ref{case1_thm_0}, $D= \{ 3 \}$ and $\ord_3 (-2) = 1$. This implies that there exist $\dfrac{1}{1} \cdot \phi(3) = 2$ cycles of length $1$ due to the elements of order $3$ in $\F_{49}^*$. According to the same theorem the elements $1$ and $-1$ (which here are denoted by $0$ and $24$) form a cycle of length $2$, while $\infty$ forms a cycle of length $1$. Moreover, being $e = 4$ according to the notation of Theorem \ref{case1_thm_3}, any cyclic element different from $1$ and $-1$ is root of a binary tree of depth $4$.     
\end{example}

\section{Case 2: $k^2 \equiv - \frac{1}{4} \pmod{p}$ with $p \equiv 1 \pmod{4}$}\label{section_case_2}
The starting point for this and the following section is \cite{sil}, Proposition 2.3.1. From (i) of the proposition one can deduce that the endomorphism ring of the elliptic curve $y^2 = x^3+x$ over $\Q$ is isomorphic to $\Z [i]$. In particular, the curve posses an endomorphism $[\alpha]$ of degree $2$, which takes a point $(x,y)$ of the curve to 
\begin{equation*}
[\alpha] (x,y) = \left(\alpha^{-2} \left( x + \frac{1}{x} \right), \alpha^{-3} y \left(1 - \frac{1}{x^2} \right) \right),
\end{equation*}
where $\alpha = 1 + \sqrt{-1}$

Let $p$ be an odd prime such that $p \equiv 1 \pmod{4}$. The elliptic curve with equation $y^2 = x^3 + x$ over $\Q$ has good reduction modulo $p$ (see \cite{sil_a}, Chapter V, Proposition 5.1 or \cite{mil}, page 59) and from now on we will denote by $E$ its reduction modulo $p$.

The quadratic equation
\begin{equation}\label{case2_eq1}
x^2-2x+2 = 0
\end{equation}
admits two distinct roots in $\F_p$, since its discriminant is equal to $-4$, which is a quadratic residue in $\F_p$. Denote by $\alpha_{\omega}$ and $\alpha_{\overline{\omega}}$ the roots of  (\ref{case2_eq1}) in $\F_p$, set $k_{\omega} = \alpha_{\omega}^{-2}$ and $k_{\overline{\omega}} =  \alpha_{\overline{\omega}}^{-2}$. We notice in passing that $k_{\omega}^2 \equiv k_{\overline{\omega}}^2 \equiv - \frac{1}{4} \pmod{p}$ and $k_{\omega} \equiv - k_{\overline{\omega}} \pmod{p}$.

Fixed a positive integer $n$ we want to study the iterations over $\Pro (\F_{p^n})$ of the maps $\theta_{k_{\sigma}}$, for $\sigma \in \{\omega, \overline{\omega} \}$.

The two maps
\begin{eqnarray*}
e_{k_{\omega}} (x,y) & = & \left(k_{\omega} \cdot \left(\frac{x^2+1}{x} \right), \frac{k_{\omega}}{\alpha_{\omega}} \cdot y \cdot \frac{x^2-1}{x^2}  \right),\\
e_{k_{\overline{\omega}}} (x,y) & = & \left(k_{\overline{\omega}} \cdot \left(\frac{x^2+1}{x} \right), \frac{k_{\overline{\omega}}}{\alpha_{\overline{\omega}}} \cdot y \cdot \frac{x^2-1}{x^2}  \right)
\end{eqnarray*}
are endomorphisms of the elliptic curve 
\begin{equation*}
E : y^2 = x^3+x
\end{equation*}
over $\F_p$.  
Hence, we can study the iterations of the maps $\theta_{k_{\omega}}$ and $\theta_{k_{\overline{\omega}}}$ taking into account the fact that 
\begin{eqnarray*}
e_{k_{\omega}} (x,y) & = & \left(\theta_{k_{\omega}} (x), \frac{k_{\omega}}{\alpha_{\omega}} \cdot y \cdot \frac{x^2-1}{x^2}  \right),\\
e_{k_{\overline{\omega}}} (x,y) & = & \left(\theta_{k_{\overline{\omega}}} (x), \frac{k_{\overline{\omega}}}{\alpha_{\overline{\omega}}} \cdot y \cdot \frac{x^2-1}{x^2}  \right).
\end{eqnarray*}

Since the endomorphism ring of the elliptic curve defined by $y^2 = x^3 +x$ over $\Q$ is isomorphic to $\Z [i]$, according to \cite{lan}, Chapter 13, Theorem 12 the endomorphism ring $\End(E)$ of $E$ over $\F_p$ is also isomorphic to $R=\Z [i]$, which is an Euclidean ring with  euclidean function
\begin{equation*}
N(a + b i) = a^2 + b^2,
\end{equation*} 
for any arbitrary choice of $a, b$ in $\Z$.

By \cite{wit}, Theorem 2.3(a), there exists an isomorphism
\begin{equation*}
\psi_n : E(\F_{p^n}) \to R/(\pi_p^n-1) R,
\end{equation*} 
where $\pi_p$ is the Frobenius endomorphism. Theorem 2.4 of \cite{wit} furnishes a representation of $\pi_p$ as an element of $R$, namely
\begin{equation*}
\pi_p = \frac{p+1-m+\sqrt{d}}{2},
\end{equation*}
where
\begin{eqnarray*}
m & = & |E(\F_p)|,\\
d & = & (p+1-m)^2-4p.
\end{eqnarray*}

Moreover, $2 = e_{k_{\omega}} \circ e_{k_{\overline{\omega}}}$, being $2$ the duplication map over the curve $E$. Since in $R$ we have that $2 = \alpha \cdot \overline{\alpha}$, the endomorphisms $e_{k_{\omega}}$ and $e_{k_{\overline{\omega}}}$ are represented in $R$ by $\alpha$ and $\overline{\alpha}$.

Fix once for the remaining part of the current section $\sigma= \omega$ or $\sigma = \overline{\omega}$.
Before studying the structure of the graph $G_{\theta_{k_{\sigma}}}^{p^n}$ we partition $\Pro(\F_{p^n})$ as follows.
\begin{enumerate}
\item If $n$ is odd and $p \equiv \pm 3 \pmod{8}$, then we define
\begin{eqnarray*}
A_n & = & \{x \in \F_{p^n} : (x, y) \in E(\F_{p^n}) \text{ for some $y \in \F_{p^n}$} \} \cup \{ \infty \};\\
B_n & = & \{x \in \F_{p^n} : (x, y) \in E(\F_{p^{2n}}) \text{ for some $y \in \F_{p^{2n}} \backslash \F_{p^n}$} \} \backslash \{1, -1 \};\\
C_n & = & \{1, -1 \}. 
\end{eqnarray*} 
\item If $n$ is even or $n$ is odd and $p \equiv \pm 1 \pmod{8}$, then we define  
\begin{eqnarray*}
A_n & = & \{x \in \F_{p^n} : (x, y) \in E(\F_{p^n}) \text{ for some $y \in \F_{p^n}$} \} \cup \{ \infty \};\\
B_n & = & \{x \in \F_{p^n} : (x, y) \in E(\F_{p^{2n}}) \text{ for some $y \in \F_{p^{2n}} \backslash \F_{p^n}$} \}.
\end{eqnarray*} 
\end{enumerate}

We notice that in both cases $A_n \cap B_n = \emptyset$. Moreover, the sets $A_n$ and $C_n$ satisfy the following properties.

\begin{lemma}\label{case1_lem_01}
The following hold.
\begin{itemize}
\item If $n$ is odd and $p \equiv \pm 3 \pmod{8}$, then $A_n \cap C_n = \emptyset$.
\item If $n$ is even or $n$ is odd and $p \equiv \pm 1 \pmod{8}$, then $\{1, -1 \} \subseteq A_n$.
\end{itemize}
\end{lemma}
\begin{proof}
From the equation $y^2 = x^3+x$ we get that
\begin{displaymath}
\begin{array}{ccc}
1 \in A_n & \Leftrightarrow & \text{$2$ is a square in $\F_{p^n}$};\\
-1 \in A_n & \Leftrightarrow & \text{$-2$ is a square in $\F_{p^n}$}.
\end{array}
\end{displaymath}

We note that $2$ is a square in $\F_p$ if and only if $\left( \frac{2}{p} \right) = 1$, while $-2$ is a square in $\F_p$ if and only if $\left( \frac{-2}{p} \right) = 1$. Since $p \equiv 1 \pmod{4}$, we get that
\begin{displaymath}
\left( \frac{2}{p} \right) = \left( \frac{-2}{p} \right) =
\begin{cases}
1 & \text{if $p \equiv \pm 1 \pmod{8}$};\\
-1 & \text{if $p \equiv \pm 3 \pmod{8}$}.
\end{cases}
\end{displaymath}
Therefore, $2$ and $-2$ are squares in $\F_p$ if and only if $p \equiv \pm 1 \pmod{8}$. In the other cases, the equations $y^2 = \pm 2$ have solutions in $\F_{p^2}$. Since $\F_{p^2} \subseteq \F_{p^n}$ if and only if $2 \mid n$, the thesis follows.
\end{proof}

The following lemmas hold.
\begin{lemma}\label{case2_lem_1}
Let $\tilde{x} \in \F_{p^n}$. Then, in $E(\F_{p^{2n}})$ there are exactly two rational points, $(\tilde{x}, \tilde{y})$ and $(\tilde{x}, -\tilde{y})$, with such an $x$-coordinate except for 
\begin{equation*}
\tilde{x} \in \{0, i_p, -i_p \},
\end{equation*}
where $i_p$ and $-i_p$ are the two square roots of $-1$ in $\F_p$, in which case $\tilde{y} = 0$ and $\tilde{x}$ is not $\theta_{k_{\sigma}}$-periodic.
\end{lemma}
\begin{proof}
The thesis is immediate if we notice that the equation $y^2 = \tilde{x}^3+\tilde{x}$ has exactly two distinct roots $y_1$ and $y_2$ in $\F_{p^{2n}}$ unless $\tilde{x}^3+\tilde{x} = 0$. As regards the last statement, we notice that $\theta_{k_{\sigma}} (\pm i_p) = 0$ and $\theta_{k_{\sigma}} (0) = \infty$. Hence, none of the three points having $y$-coordinate equal to $0$ is $\theta_{k_{\sigma}}$-periodic.
\end{proof}

\begin{lemma}\label{case2_lem_1_2}
The map $\theta_{k_{\sigma}}$ takes the elements of $A_n$ to $A_n$, the elements of $B_n$ to $B_n$ and the elements of $C_n$ to $A_n$.
\end{lemma}
\begin{proof}
If $\tilde{x} = \infty$, then $\theta_{k_{\sigma}} (\tilde{x}) = \infty$. If $\tilde{x} \in A_n \backslash\{ \infty \}$, then there exists $\tilde{y} \in \F_{p^n}$ such that $(\tilde{x},\tilde{y}) \in E(\F_{p^n})$. Therefore, $e_{k_{\sigma}} (\tilde{x}, \tilde{y}) \in E(\F_{p^n})$ and  $\theta_{k_{\sigma}}(\tilde{x}) \in A_n$.  

If $\tilde{x} \in B_n$, then there exists $\tilde{y} \in \F_{p^{2n}} \backslash \F_{p^n}$ such that $(\tilde{x}, \tilde{y}) \in E (\F_{p^{2n}})$. We notice that $(\theta_{k_{\sigma}}(\tilde{x}), y_1)$ and $(\theta_{k_{\sigma}}(\tilde{x}), y_2)$, with
\begin{eqnarray*}
y_j & = & (-1)^j \cdot \frac{k_{\sigma}}{\alpha_{\sigma}} \cdot \tilde{y} \cdot \frac{\tilde{x}^2-1}{\tilde{x}^2}, \quad \text{for $j \in \{1, 2 \}$,}  
\end{eqnarray*}  
are the only rational points in $E(\F_{p^{2n}})$ having $\theta_{k_{\sigma}} (\tilde{x})$ as $x$-coordinate. Since $y_j \in \F_{p^n}$ if and only if $\tilde{x}^2 = 1$, because $\tilde{y} \not \in \F_{p^n}$,  and $\{1, -1 \} \not \in B_n$, we conclude that $\theta_{k_{\sigma}} (\tilde{x}) \in B_n$.  

As regards the elements of $C_n$ we notice that 
\begin{equation*}
\theta_{k_{\sigma}} (\pm 1) = \pm 2 k_{\sigma} \in \{ \pm i_p \} \subseteq A_n.
\end{equation*}
\end{proof}

According to Lemma \ref{case2_lem_1_2} we can study the iterations of $\theta_{k_{\sigma}}$ on $\Pro(\F_{p^n})$ separately on the elements of $A_n$ and $B_n$. 
Let us firstly define
\begin{displaymath}
\rho_0 = 
\begin{cases}
\alpha, & \text{if $\alpha^{-2} \equiv k_{\sigma} \pmod{\pi_p}$};\\
\overline{\alpha}, & \text{if $\overline{\alpha}^{-2} \equiv k_{\sigma} \pmod{\pi_p}$}.
\end{cases}
\end{displaymath}

Since $E(\F_{p^n})$ is isomorphic to $R/(\pi_p^n-1) R$ via the isomorphism given by $\psi_n$, the iterations of $\theta_{k_{\sigma}}$ on $A_n$ can be studied relying upon the iterations of $[\rho_0]$ in $R/(\pi_p^n-1)R$.

Now we define
\begin{eqnarray*}
E(\F_{p^{2n}})_{B_n} = \left\{(x,y) \in E(\F_{p^{2n}}) : x \in B_n \right\}
\end{eqnarray*}
and denote by $\lambda, \tau$ two elements of $\F_{p^2}$ such that
\begin{eqnarray*}
(\pm \lambda)^2 & = & 2;\\
(\pm \tau)^2 & = & -2. 
\end{eqnarray*}

If $n$ is odd and $p \equiv \pm 3 \pmod{8}$, then we define 
\begin{equation*}
E(\F_{p^{2n}})_{B_n}^* = \{ O , (0,0), (\pm i_p, 0), (1, \pm \lambda), (-1, \pm \tau) \}.
\end{equation*}

In all other cases we define
\begin{equation*}
E(\F_{p^{2n}})_{B_n}^* = \{ O , (0,0), (\pm i_p, 0) \}.
\end{equation*}
In both cases $O$ denotes the point at infinity.

The following holds.
\begin{lemma}\label{case2_lem_2}
Let $\tilde{x} \in \F_{p^{2n}}$ and $P = (\tilde{x}, \tilde{y}) \in E(\F_{p^{2n}})$. We have that $(\pi_p^n+1) P = O$ if and only if $P \in E(\F_{p^{2n}})_{B_n} \cup E(\F_{p^{2n}})_{B_n}^*$.
\end{lemma}
\begin{proof}
Suppose that $(\pi_p^n+1) P = O$. If $\tilde{y} = 0$, then $P \in E(\F_{p^{2n}})_{B_n}^*$. Suppose on the contrary that $\tilde{y} \not = 0$. Since $\tilde{y} \in \F_{p^n}$ only if $\tilde{y}^{p^n} = \tilde{y}$ and by hypothesis $\tilde{y}^{p^n} = -\tilde{y}$, it follows that $\tilde{y} \in \F_{p^n}$ only if $\tilde{y} = -\tilde{y}$, namely only if $2 \tilde{y}=0$. Since $\tilde{y} \not =0$, we conclude that $\tilde{y} \in \F_{p^{2n}} \backslash \F_{p^n}$. In this latter case $P \in E(\F_{p^{2n}})_{B_n}$, unless $n$ is odd, $p \equiv \pm 3 \pmod{8}$ and $\tilde{x} \in \{1, -1 \}$, in which case $P \in E(\F_{p^{2n}})_{B_n}^*$. 

Vice versa, suppose that $P \in E(\F_{p^{2n}})_{B_n}$. Then, $\pi_p^n (P) = (x, y^{p^n})$ and $y^{p^{n}} \not = y$, since $y \in \F_{p^{2n}} \backslash \F_{p^n}$. Therefore, $\pi_p^n (P) = - P$. Finally, if $P \in E(\F_{p^{2n}})_{B_n}^*$, then we can check by direct computation that $(\pi_p^n+1) P = O$. 
\end{proof}

With the notation till now introduced and in virtue of Lemma \ref{case2_lem_2} we can say that there exists an isomorphism
\begin{equation*}
\widetilde{\psi}_n : E(\F_{p^{2n}})_{B_n} \cup E(\F_{p^{2n}})_{B_n}^*  \to R / (\pi_p^n+1) R.
\end{equation*}

Hence, the iterations of $\theta_{k_{\sigma}}$ on $B_n$ can be studied by means of the iterations of $[\rho_0]$ on $R / (\pi_p^n+1) R$.

Suppose that  $\pi_p^n-1$ (resp. $\pi_p^n+1$) factors in $R$, up to units, as
\begin{equation}\label{case2_eq_3}
\rho_0^{e_0} \cdot \left( \prod_{i=1}^v p_i^{e_i} \right) \cdot \left( \prod_{i = v+1}^w {r_i}^{e_i} \right),
\end{equation}
where 
\begin{enumerate}
\item any $e_i$ is a non-negative integer, for $0 \leq i \leq w$;
\item $N(\rho_0^{e_{0}}) = 2^{e_{0}}$;
\item for $1 \leq i \leq v$ the elements $p_i \in \Z$ are distinct primes of $R$ and $N( p_i^{e_i} ) = p_i^{2 e_i}$;
\item for $v+1 \leq i \leq w$ the elements $r_i \in R \backslash \Z$  are distinct primes of $R$, different from $\rho_0$ and $\overline{\rho}_0$,  and $N( r_i^{e_i} ) = p_i^{e_i}$, for some rational integer $p_i$ such that $r_i \overline{r}_i = p_i$.
\end{enumerate}

For the sake of clarity we define, for $1 \leq i \leq w$,  
\begin{displaymath}
\rho_i = 
\begin{cases}
p_i, & \text{if $1 \leq i \leq v$};\\
r_i, & \text{if $v+1 \leq i \leq w$}.
\end{cases}
\end{displaymath}
As a consequence of the factorization (\ref{case2_eq_3}) the ring $R / (\pi_p^n - 1) R$ (resp. $R / (\pi_p^n + 1) R$) is isomorphic to 
\begin{equation}\label{case2_eq_2}
S = \prod_{i = 0}^w R / \rho_i^{e_i} R.
\end{equation}

As regards the additive structure of the quotient rings involved in (\ref{case2_eq_2}), we notice the following.
\begin{itemize}
\item The additive group of $R / \rho_0^{e_{0}} R$ is cyclic of orders $2^{e_{0}}$. Hence, there are $\phi(2^{h_{0}})$ elements in $R / \rho_0^{e_{0}} R$  of order $2^{h_{0}}$, for each integer $0 \leq h_{0} \leq e_{0}$. 

\item For any $i \in \{1, \dots, v \}$ the additive group of $R / p_i^{e_i} R$ is isomorphic to the direct sum of two cyclic groups of order $p_i^{e_i}$. This implies that, for each integer $0 \leq h_i \leq e_i$, there are $N_{h_i}$ elements in $R / p_i^{e_i} R$ of order $p_i^{h_i}$, where
\begin{displaymath}
N_{h_i} = \left\{
\begin{array}{ll}
1 & \text{ if $h_i = 0$,}\\
{p_i}^{2 h_i} - {p_i}^{2(h_{i}-1)} & \text{ otherwise.}
\end{array}
\right.
\end{displaymath}

\item For any $i \in \{v+1, \dots, w \}$ the additive group of $R / r_i^{e_i} R$ is cyclic of order $p_i^{e_i}$.  Hence, there are $\phi(p_i^{h_i})$ elements in $R / r_i^{e_i} R$ of order $p_i^{h_i}$, for each integer $0 \leq h_i \leq e_i$. 

\end{itemize}

If $(x,y)$ is a rational point of $E(\F_{p^n})$ (resp. $E(\F_{p^{2n}})_{B_n}$), then we write $P_{(x,y)}$ for the image of $(x,y)$ in $S$.

Now we define the sets 
\begin{eqnarray*}
Z_{i} & = & \{0, 1, \dots, e_i \}, \quad \text{for any $0 \leq i \leq w$}
\end{eqnarray*}
and
\begin{equation*}
H = \prod_{i = 0}^w Z_{i}.
\end{equation*} 

If $P = (P_{0}, P_1, \dots, P_w) \in S$, then we define $h^P = (h_{0}^P, h_{1}^P, \dots, h_w^P)$ in $H$ if 
\begin{itemize}
\item $P_{0}$ has additive order $2^{h_{0}^P}$ in $R / \rho_0^{e_0} R$;
\item each $P_i$, for $1 \leq i \leq w$, has additive order $p_i^{h_i^P}$ in $R / \rho_i^{e_i} R$.
\end{itemize}

Moreover, we define $o(P) = \left(o(P_{0}), o(P_{1}), \dots, o(P_w) \right)$, where $o(P_i)$ denotes, for any $0 \leq i \leq w$,  the additive order of $P_i$ in $R / \rho_i^{e_i} R$.

The following two lemmas furnish a characterization of $\theta_{k_{\sigma}}$-periodic elements.

\begin{lemma}\label{case2_lem_7}
Let $\tilde{x} \in A_n$ (resp. $B_n$) be $\theta_{k_{\sigma}}$-periodic. Then, one of the following holds:
\begin{itemize}
\item $\tilde{x} = \infty$;
\item $\tilde{x}$ is the $x$-coordinate of a rational point $(\tilde{x},\tilde{y}) \in E(\F_{p^n})$ (resp. $E(\F_{p^{2n}})_{B_n}$) such that $P_{(\tilde{x},\tilde{y})} = (P_{0}, P_{1}, \dots, P_w)$, where $P_{0} = 0$.
\end{itemize}

\end{lemma}
\begin{proof}
If $\tilde{x} \in A_n \backslash \{ \infty \}$ (resp. $B_n$), then $(\tilde{x},\tilde{y}) \in E(\F_{p^n})$ (resp. $E(\F_{p^{2n}})_{B_n}$) for some $\tilde{y}$ in $\F_{p^n}$ (resp. $\F_{p^{2n}}$). Moreover, $[\rho_0]^l P_{(\tilde{x},\tilde{y})} = \pm P_{(\tilde{x},\tilde{y})}$ for some positive integer $l$, namely $[\rho_0^l \pm 1] P_{0} = 0$. Since $\rho_0$ does not divide $\rho_0^l \pm 1$ in $R$, we have that $P_{0} = 0$.
\end{proof}

\begin{lemma}\label{case2_lem_5}
Let $P = (P_{0}, P_{1}, \dots, P_w)$ be a point in $S$ such that $P_{0} = 0$ and denote  by $l_i$, for any $0 \leq i \leq w$, the smallest among the positive integers $s$ such that either $[\rho_0]^s \cdot P_{i} = P_{i}$ or $[\rho_0]^s \cdot P_i = -P_{i}$.

Let $l' = \lcm (l_{0}, l_{1}, \dots, l_w)$ and denote by $l$ the smallest among the positive integers $s$ such that $[\rho_0]^s P = P$ or $[\rho_0]^s P = -P$.

Then, any $l_{i}$ is the smallest among the positive integers $s$ such that $\rho_i^{h_{i}^P}$ divides either $\rho_0^{s} + 1$ or  $\rho_0^{s} - 1$ in $R$. Moreover, $l$ is determined as follows: 
\begin{displaymath}
l = 
\begin{cases}
l' & \text{if either $\rho_i^{h_{i}^P} \mid (\rho_0^{l'} +1) $ for all $i$ or $\rho_i^{h_{i}^P} \mid (\rho_0^{l'} - 1) $ for all $i$},\\
2 l' & \text{otherwise}.
\end{cases}
\end{displaymath}
\end{lemma}
\begin{proof}
Fix a $i$ such that $0 \leq i \leq w$. If $P_i = 0$, then $h_i^P=0$, $l_i = 1$ and $\rho_i^0 =1$ divides both $\rho_0^{l_i}+1$ and $\rho_0^{l_i}-1$. Suppose now that $P_i \not = 0$. We notice that $o(P_i) \cdot P_i = 0$ and that, whichever $i$ is, $\gcd(\rho_i^{e_i}, o(P_i)) = \rho_i^{h_i^P}$. Therefore,
\begin{equation*}
\frac{\rho_i^{e_i}}{\gcd(\rho_i^{e_i}, o(P_i))} = \rho_i^{e_i-h_i^P} \mid P_i.
\end{equation*}
Hence, if $s$ is a positive integer such that $\rho_i^{h_i^P}$ divides either $\rho_0^s+1$ or $\rho_0^s-1$, then $\rho_i^{e_i} = \rho_i^{h_i^P} \cdot \rho_i^{e_i-h_i^P}$ divides either $(\rho_0^s+1) P_i$ or $(\rho_0^s-1) P_i$,
namely 
\begin{equation*}
[\rho_0]^s \cdot P_i = P_i \quad \text{or} \quad [\rho_0]^s \cdot P_i = - P_i. 
\end{equation*}
We can therefore conclude that $l_i$ is also the smallest among the positive integers $s$ such that $\rho_i^{h_{i}^P}$ divides either $\rho_0^{s} + 1$ or  $\rho_0^{s} - 1$ in $R$.

Since $[\rho_0]^l P = P$ or $[\rho_0]^l P = -P$, we have that either $[\rho_0]^l P_i = P_i$ for all $i$ or $[\rho_0]^l P_i = -P_i$ for all $i$. Then, $l$ must be a common multiple of all $l_i$. In fact, suppose that for a chosen integer $i$ such that $0 \leq i \leq w$
\begin{displaymath}
\begin{cases}
l = l_i \cdot q_i + t_i\\
0 \leq t_i < l_i
\end{cases}
\end{displaymath}
for integers $q_i, t_i$. Then,
\begin{equation*}
[\rho_0]^{l} P_i = [\rho_0]^{l_i \cdot q_i + t_i} P_i = [\rho_0]^{t_i} (\pm P_i) = \pm P_i.
\end{equation*}
By definition of $l_i$ we conclude that $t_i=0$. Hence, $l_i \mid l$ and $l' \mid l$. Since $[\sigma]^{l'} P_i = \pm P_i$ for any $i$, we get the thesis. 
\end{proof}

The following holds.
\begin{lemma}\label{case2_lem_6}
Let $x_0 \in \F_{p^n}$ be $\theta_{k_{\sigma}}$-periodic, $s$ a positive integer and $x_s = \theta_{k_{\sigma}}^s (x_0)$. Let $(x_0,y_0) \in E(\F_{p^n})$ (resp. $E(\F_{p^{2n}})_{B_n}$) for some $y_0 \in \F_{p^n}$ (resp. $\F_{p^{2n}}$) and $(x_s, y_s) \in E(\F_{p^n})$ (resp. $E(\F_{p^{2n}})_{B_n}$) for some $y_s \in \F_{p^n}$ (resp. $\F_{p^{2n}}$). If $Q^{(0)} = P_{(x_0,y_0)}$ and $Q^{(s)} = P_{(x_s,y_s)}$, then $h^{Q^{(0)}} = h^{Q^{(s)}}$.
\end{lemma}
\begin{proof}
Since $x_s = \theta_{k_{\sigma}}^s (x_0)$, we deduce that $Q^{(s)} = \pm [\rho_0]^s Q^{(0)}$. Therefore, if $Q^{(0)} = (Q_{0}^{(0)}, Q_{1}^{(0)}, \dots, Q_w^{(0)})$, then $Q^{(s)} = \pm ([\rho_0]^s Q_{0}^{(0)}, \dots, [\rho_0]^s Q_w^{(0)})$. For any $i \not = 0$ we have that $\rho_0^s$ and $\rho_i$ are coprime. Therefore, $h_i^{Q^{(0)}} = h_i^{Q^{(s)}}$ for any $i \not = 0$. Finally, $Q_{0}^{(0)} = 0$, implying that $h^{Q^{(0)}}_{0} = h^{Q^{(s)}}_{0} = 0$. Therefore we can conclude that $h^{Q^{(0)}} = h^{Q^{(s)}}$.
\end{proof}

We introduce the following notation.

\begin{definition}
If $x \in \Pro (\F_{p^n})$ is $\theta_{k_{\sigma}}$-periodic, then we denote by
\begin{equation*}
\left< x \right>_{\theta_{k_{\sigma}}}  = \{ \theta_{k_{\sigma}}^r (x) : r \in \N \} 
\end{equation*}
the elements of the cycle of $x$ with respect to the map $\theta_{k_{\sigma}}$.
\end{definition}

In virtue of Lemmas \ref{case2_lem_5} and \ref{case2_lem_6} we can give the following definition.
\begin{definition}
If $h = (h_{0}, h_{1}, \dots, h_w) \in H$, then $C_h$ denotes the set of all cycles $\left< x \right>_{\theta_{k_{\sigma}}}$ of $G^{p^n}_{\theta_{k_{\sigma}}}$ where $x \in A_n$ (resp. $B_n$) and exactly one of the following conditions holds:
\begin{itemize}
\item $x = \infty$ (and $h=(0,0, \dots, 0)$);
\item $(x, y) \in E(\F_{p^n})$ (resp. $E (\F_{p^{2n}})_{B_n}$) for some $y \in \F_{p^n}$ (resp. $\F_{p^{2n}}$) and $h^{P_{(x,y)}} = h$. 
\end{itemize}
Moreover, we introduce the following notations:
\begin{itemize}
\item $l_h$ is the length of the cycles in $C_h$;
\item $C_{A_n} = \{\langle x \rangle_{\theta_{k_{\sigma}}} \in G^{p^n}_{\theta_{k_{\sigma}}} : x \in A_n \}$;
\item $C_{B_n} = \{\langle x \rangle_{\theta_{k_{\sigma}}} \in G^{p^n}_{\theta_{k_{\sigma}}} : x \in B_n \}$;
\item $N_h  =  \dfrac{1}{2 l_h} \cdot \phi(2^{h_0}) \cdot \left( \displaystyle\prod_{i=1}^{v} N_{h_i} \right)  \cdot \left( \displaystyle\prod_{i=v+1}^{w} \phi (p_i^{h_i}) \right)$.
\end{itemize}
\end{definition}

Now we state two theorems, concerning respectively the cycles of $C_{A_n}$ and $C_{B_n}$.
\begin{theorem}\label{case2_thm_1_A}
Let $S \cong R / (\pi_p^n-1) R$ and let $H_A \subseteq H$ be the set formed by all $h \in H$ such that $h_0 = 0$. Then, 
\begin{equation*}
C_{A_n} = \bigsqcup_{h \in H_A} C_h.
\end{equation*}

Moreover, $|C_h|= 1$ if $h_i=0$ for all $0 \leq i \leq w$, while $|C_h| = N_h$.
\end{theorem}

\begin{proof}
Since $C_h \subseteq C_{A_n}$ for any $h \in H_A$, we have that $\displaystyle\bigsqcup_{h \in H_A} C_h \subseteq C_{A_n}$. Vice versa, if $\langle \tilde{x} \rangle_{\theta_{k_{\sigma}}} \in C_{A_n}$, then either $\tilde{x} = \infty$, or $(\tilde{x}, \tilde{y}) \in E(\F_{p^n})$ for some $\tilde{y} \in \F_{p^n}$ and $P = P_{(\tilde{x}, \tilde{y})} = (0, P_1, \dots, P_w)$ by Lemma \ref{case2_lem_7}. Hence, $h_0^P = 0$ and $h^P \in H_A$. Therefore, $C_{A_n} \subseteq \displaystyle\bigsqcup_{h \in H_A} C_h$.

If $\left< \tilde{x} \right>_{\theta_{k_{\sigma}}} \in C_h$, where $h \in H_A$ with all $h_i = 0$, then $h=h^P$, where all $P_i = 0$. Therefore, $ \tilde{x} = \infty$ and $|C_h| = 1$. 

Consider now $\left< \tilde{x} \right>_{\theta_{k_{\sigma}}} \in C_h$ for some $h \in H_A \backslash \{(0, \dots, 0) \}$. Then, $(\tilde{x}, \tilde{y}) \in E(\F_{p^n})$ for some $\tilde{y} \in \F_{p^n}$ and the additive order of $P_{(\tilde{x}, \tilde{y})}$ is not $2$. Hence, there are two distinct rational points in $E(\F_{p^n})$ with such a $x$-coordinate. Since the length of the cycle $\left< \tilde{x} \right>_{\theta_{k_{\sigma}}}$ is $l_h$ and the number of points $Q$  in $S$ such that $h^Q = h$ is 
\begin{equation*}
\phi(2^{h_{0}}) \cdot \left( \displaystyle\prod_{i=1}^{v} N_{h_i} \right) \cdot \left( \displaystyle\prod_{i=v+1}^{w} \phi(p_i^{h_i}) \right),
\end{equation*}
the thesis follows.

\end{proof} 

\begin{theorem}\label{case2_thm_1_B}
Let $S \cong R / (\pi_p^n+1) R$ and let $H_B \subseteq H$ be the set formed by all $h \in H$ such that $h_0 = 0$ and $h_i \not =0$ for some $1 \leq i \leq w$. Then,
\begin{equation*}
C_{B_n} = \bigsqcup_{h \in H_B} C_h
\end{equation*}
and $|C_h| = N_h$ for any $h \in H_B$.
\end{theorem}

\begin{proof}
Since $C_h \subseteq C_{B_n}$ for any $h \in H_B$, we have that $\displaystyle\bigsqcup_{h \in H_B} C_h \subseteq C_{B_n}$. Vice versa, if $\langle \tilde{x} \rangle_{\theta_{k_{\sigma}}} \in C_{B_n}$, then $\tilde{x} \in \F_{p^n}$ and $(\tilde{x}, \tilde{y}) \in E(\F_{p^{2n}})_{B_n}$ for some $\tilde{y} \in \F_{p^{2n}} \backslash \F_{p^n}$. In particular, $P = P_{(\tilde{x}, \tilde{y})}$ is not the point at infinity and $h^P \not = (0, 0, \dots, 0)$. At the same time $h_0^P =0$ by Lemma \ref{case2_lem_7} and consequently $h_i^P \not = 0$ for some $1 \leq i \leq w$. Therefore, $C_{B_n} \subseteq \displaystyle\bigsqcup_{h \in H_B} C_h$.

Consider now $\left< \tilde{x} \right>_{\theta_{k_{\sigma}}} \in C_h$ for some $h \in H_B$. Then, $(\tilde{x}, \tilde{y}) \in E(\F_{p^{2n}})$ for some $\tilde{y} \in \F_{p^{2n}}$. Moreover, by definition of $H_B$, the additive order of $P_{(\tilde{x}, \tilde{y})}$ is not $2$. Hence, there are two distinct rational points in $E(\F_{p^{2n}})_{B_n}$ having $\tilde{x}$ as $x$-coordinate. Since the length of the cycle $\left< \tilde{x} \right>_{\theta_{k_{\sigma}}}$ is $l_h$ and the number of points $Q$  in $S$ such that $h^Q = h$ is 
\begin{equation*}
\phi(2^{h_{0}}) \cdot \left( \displaystyle\prod_{i=1}^{v} N_{h_i} \right) \cdot \left( \displaystyle\prod_{i=v+1}^{w} \phi(p_i^{h_i}) \right),
\end{equation*}
the thesis follows.
\end{proof} 

In the following we will denote by $V_{A_n}$ (resp. $V_{B_n}$) the set of the $\theta_{k_{\sigma}}$-periodic elements of $A_n$ (resp. $B_n$).  Before proceeding with the description of the trees rooted in vertices of $V_{A_n}$ and $V_{B_n}$ we notice that,  according to \cite{gil},
\begin{eqnarray*}
R / \rho_0^{e_{0}} R & = & \left\{ \sum_{i=0}^{e_{0}-1} \delta_i \cdot [\rho_0]^i: \delta_i = 0 \text{ or } 1 \right\}.
\end{eqnarray*}

The following theorem characterizes the reversed trees having root in $V_{A_n}$ (resp. $V_{B_n}$).
\begin{theorem}\label{case2_thm_2}
Any element $x_0 \in V_{A_n}$ (resp. $V_{B_n}$) is the root of a reversed binary tree having the following properties.
\begin{itemize}
\item If $x_0 \not = \infty$, then the depth of the tree is $e_0$, the root has exactly one child, while all other vertices have two distinct children.
\item If $x_0 = \infty$, $n$ is odd and $p \equiv \pm 3 \pmod{8}$, then the tree has depth $3$, the root and the two vertices at the level $2$ have exactly one child each, while the vertex at the level $1$ has two distinct children.  
\item If $x_0 = \infty$, $n$ is odd and $p \equiv \pm 1 \pmod{8}$ or $n$ is even, then the tree has depth $e_0$, the root and the two vertices at the level $2$ have exactly one child each, while all other vertices have two distinct children. 
\end{itemize}
\end{theorem}
\begin{proof}
At first we notice that any vertex of the tree rooted in $x_0$ has at most two children. In fact, consider a vertex $x_r$ at the level $r \geq 0$ of such a tree. If $x_r = \infty$, then $\theta_{k_{\sigma}} (x) = x_r$ if and only if $x \in \{0, \infty \}$. Since $\infty$ is $\theta_{k_{\sigma}}$-periodic, it follows that $\infty$ has exactly one child. If $x_r \not = \infty$, then  $\theta_{k_{\sigma}} (x)  = x_r$ if and only if $x^2 - \frac{x_r}{k_{\sigma}}x+1=0$. This quadratic equation has exactly one root if its discriminant $\frac{x_r^2}{k_{\sigma}^2}-4$ is zero. This happens if and only if $x_r = \pm 2 k_{\sigma}$. Since $(\pm 2 k_{\sigma})^2 = -1$, we deduce that the discriminant is zero if and only if $x_r \in \{ i_p, - i_p \}$. We notice in passing that $i_p$ and $-i_p$ have one child each, namely $1$ and $-1$. Moreover, $\theta_{k_{\sigma}} (\pm i_p) = 0$ and $\theta_{k_{\sigma}} (0) = \infty$. Therefore we can conclude that, provided that $x_0 \not = \infty$, the root $x_0$ has exactly one child, while all other vertices have zero or two distinct children. If $x_0 = \infty$, then $x_0$ has exactly one child, any other vertex $x_r$ of the tree has zero or two distinct children, unless $x_r$ is one of the two vertices at the level $2$ of the tree having exactly one child each. 

According to Lemma \ref{case2_lem_1_2} the map $\theta_{k_{\sigma}}$ takes the elements of $A_n$ to $A_n$, the elements of $B_n$ to $B_n$ and the elements of $C_n$ to $A_n$. The elements of $C_n$, namely $1$ and $-1$, are the only two vertices belonging to the level $3$ of the tree rooted in $x_0 = \infty$ in the case that $n$ is odd and $p \equiv \pm 3 \pmod{8}$. Under such hypotheses $1$ and $-1$ are also the leaves of the tree rooted in $\infty$.
In all other circumstances, if $x_0 \in V_{A_n}$ (resp. $V_{B_n}$), then $x_0$ is root of a tree whose vertices belong to $A_n$ (resp. $B_n$) and whose depth is $e_0$.    
In fact, $x_0 = \infty$ or is the $x$-coordinate of a point $(x_0, y_0) \in E(\F_{p^n})$ (resp. $E(\F_{p^{2n}})_{B_n}$). Let $P=(0, P_1, \dots, P_w)$ be the image of the point at infinity or $(x_0,y_0)$ in $S = \displaystyle\prod_{i=0}^w R / \rho_i^{e_i} R$. If $Q = (Q_0, Q_1, \dots, Q_w)$ is any point of $S$, then $[\rho_0]^{e_0} Q_0 = 0$ in $R / \rho_0^{e_0} R$. Therefore, $[\rho_0]^{e_0} Q$ is the image in $S$ of a rational point of $E(\F_{p^n})$ (resp. $E(\F_{p^{2n}})_{B_n}$) whose $x$-coordinate is $\theta_{k_{\sigma}}$-periodic. This fact implies that $x_0$ cannot be the root of a tree of depth greater than $e_0$. 
Suppose now that
\begin{eqnarray*}
Q_0 & = & [1];\\
Q_i & = & [\rho_0]^{-e_0} P_i, \text{for any $1 \leq i \leq w$}.
\end{eqnarray*}
Then, $[\rho_0]^r Q_0 \not = 0$ in $R/ \rho_0^{e_0} R$ for any $0 \leq r < e_0$. Nonetheless, $[\rho_0]^{e_0} Q_0 = 0$ in $R / \rho_0^{e_0} R$ and $[\rho_0]^{e_0} Q_i = P_i$ for any $1 \leq i \leq w$. Therefore, $[\rho_0]^{e_0} Q = P$. If $Q$ is the image in $S$ of a rational point $(\tilde{x}, \tilde{y}) \in E(\F_{p^n})$ (resp. $E({\F_{p^{2n}}})_{B_n}$), then $\theta_{k_{\sigma}}^{e_0} (\tilde{x}) = x_0$ and $\theta_{k_{\sigma}}^r (\tilde{x})$ is not $\theta_{k_{\sigma}}$-periodic for any $0 \leq r < e_0$, namely $\tilde{x}$ belongs to the level $e_0$ of the tree rooted in $x_0$. 

Consider now a vertex $x_r \not = \infty$ at the level $r < e_0$ of the tree rooted in $x_0$. Then, $(x_r, y_r) \in E(\F_{p^n})$ (resp. $E(\F_{p^{2n}})_{B_n}$) for some $y_r \in \F_{p^{n}}$ (resp. $\F_{p^{2n}} \backslash \F_{p^n}$). If $Q = Q_{(x_r, y_r)} = (Q_0, Q_1, \dots, Q_w) \in S$, then 
\begin{equation*}
Q_0 = [\rho_0]^{e_0-r} + \sum_{i=e_0-r+1}^{e_0-1} \delta_i \cdot [\rho_0]^i
\end{equation*} 
for some choice of the coefficients $\delta_i \in \{0, 1 \}$. Now we prove that there exists at least a rational point $(x_{r+1}, y_{r+1}) \in E(\F_{p^n})$ (resp. $E(\F_{p^{2n}})_{B_n}$) such that $\theta_{k_{\sigma}} (x_{r+1}) = x_r$. Consider in fact the point $\tilde{Q} = \tilde{Q}_{(x_{r+1}, y_{r+1})}= (\tilde{Q}_0, \tilde{Q}_1, \dots, \tilde{Q}_w) \in S$, where
\begin{eqnarray*}
\tilde{Q}_0 & = & [\rho_0]^{e_0-r-1} + \sum_{i=e_0-r}^{e_0-2} \delta_{i+1} \cdot [\rho_0]^i;\\
\tilde{Q}_i & = & [\rho_0]^{-1} Q_i, \quad \text{for $1 \leq i \leq w$.} 
\end{eqnarray*}
Since $[\rho_0] \tilde{Q} = Q$ we are done. Finally, $\infty$ has exactly one child. 

All considered, we have proved that any vertex $x_r$ at the level $r < e_0$ of the tree rooted in a vertex $x_0$ of $V_{A_n}$ (resp. $V_{B_n}$) has exactly two children, unless $x_r \in \{x_0, i_p, - i_p \}$, in which case $x_r$ has exactly one child.
\end{proof}

We conclude this section with one example.

\begin{example} 
Let $p=53$. Then, the two roots in $\F_p$ of the equation
\begin{equation*}
x^2-2x+2=0
\end{equation*}
are
\begin{equation*}
\alpha_{\omega} = 24 \text{ and } \alpha_{\overline{\omega}} = 31.
\end{equation*}
Consequently,
\begin{equation*}
k_{\omega} = 15 \text{ and } k_{\overline{\omega}} = 38.
\end{equation*}

In the following we will describe thoroughly the structure of the graphs $G_{\theta_{15}}^{53}$ and $G_{\theta_{38}}^{53}$. At first we notice that $|E(\F_{53})| = 68$ and that the representation of the Frobenius endomorphism $\pi_{53}$ as an element of $R$ is
\begin{equation*}
\pi_{53} = -7+2i.
\end{equation*}

We deal firstly with the graph $G_{\theta_{15}}^{53}$. Since $\overline{\alpha}^{-2} \equiv 15 \pmod{\pi_{53}}$, we define $\rho_0 = \overline{\alpha}$ and study the iterations of the map $\theta_{15}$ on $A_1$ by means of the iterations of $[\rho_0]$ on $S \cong R / (\pi_{53}-1) R$. We have that
\begin{equation*}
S = R / \rho_0^2 R \times R / \rho_1 R,
\end{equation*}
where $\rho_1 = -1-4i$. Since $e_0 = 2$, the depth of the trees rooted in elements of $A_1 \backslash \{ \infty \}$ is $2$, while the depth of the tree rooted in $\infty$ is $3$, because $p \equiv -3 \pmod{8}$ and $n=1$. The $\theta_{15}$-periodic elements in $A_1$ are $x$-coordinates of the points $P=(0,P_1) \in S$. In $R/\rho_1 R$ there are $16$ points of additive order $17$ and $1$ point of additive order $1$, namely $P_1 = 0$. In this latter case $P=(0,0)$, namely $P$ is the point at infinity, which gives rise to one cycle of length $1$. In the former case the $16$ points give rise to a cycle of length $8$, since $8$ is the smallest among the positive integers $s$ such that $\rho_1$ divides either $\rho_0^s+1$ or $\rho_0^s-1$.

The iterations of $\theta_{15}$ on $B_1$ can be studied by means of the iterations of $[\rho_0]$ on $S \cong R / (\pi_{53}+1) R$. We have that
\begin{equation*}
S = R / \rho_0^3 R \times R / \rho_1 R,
\end{equation*}
where $\rho_1 = 1-2i$. Since $e_0 = 3$, the depth of the trees rooted in elements of $B_1$ is $3$. The $\theta_{15}$-periodic elements in $B_1$ are $x$-coordinates of the points $P=(0,P_1) \in S$, where $P_1 \not = 0$. The $4$ points $P_1 \not = 0$ in $R/\rho_1 R$ have additive order $5$. They give rise to one cycle of length $2$, since $2$ is the smallest among the positive integers $s$ such that $\rho_1$ divides either $\rho_0^s+1$ or $\rho_0^s-1$. 

The graph $G_{\theta_{15}}^{53}$ is represented below. We notice that the vertex labels, different from $\infty$ and `0' (the zero in $\F_{53}$), refer to the exponents of the powers $\gamma^i$ for $0 \leq i \leq 51$, being $\gamma$ the root of the Conway polynomial $x-2 \in \F_{53} [x]$. 

The following $2$ connected components are due to the elements of $A_1$. 
\begin{center}
\begin{picture}(90, 90)(-28,-45)
	\unitlength=2.8pt
    \gasset{Nw=4,Nh=4,Nmr=3,curvedepth=0}
    \thinlines
   \footnotesize
    \node(A1)(15,0){24}
    \node(A2)(10.6,10.6){46}
    \node(A3)(0,15){10}
    \node(A4)(-10.6,10.6){44}
    \node(A5)(-15,0){50}
    \node(A6)(-10.6,-10.6){20}
    \node(A7)(0,-15){36}
    \node(A8)(10.6,-10.6){18}
    
    \node(B1)(30,0){34}
    \node(B2)(21.2,21.2){28}
    \node(B3)(0,30){6}
    \node(B4)(-21.2,21.2){42}
    \node(B5)(-30,0){8}
    \node(B6)(-21.2,-21.2){2}
    \node(B7)(0,-30){32}
    \node(B8)(21.2,-21.2){16}
    
    \node(C12)(44.1,8.8){33}
    \node(C21)(37.4,25){3}
    \node(C22)(25,37.4){49}
    \node(C31)(8.8,44.1){1}
    \node(C32)(-8.8,44.1){51}
    \node(C41)(-25,37.4){5}
    \node(C42)(-37.4,25){47}
    \node(C51)(-44.1,8.8){7}
    \node(C52)(-44.1,-8.8){45}
    \node(C61)(-37.4,-25){23}
    \node(C62)(-25,-37.4){29}
    \node(C71)(-8.8,-44.1){25}
    \node(C72)(8.8,-44.1){27}
    \node(C81)(25,-37.4){21}
    \node(C82)(37.4,-25){31}
    \node(C11)(44.1,-8.8){19}
 
    \drawedge(A1,A2){}
    \drawedge(A2,A3){}
    \drawedge(A3,A4){}
    \drawedge(A4,A5){}
    \drawedge(A5,A6){}
    \drawedge(A6,A7){}
    \drawedge(A7,A8){}
    \drawedge(A8,A1){}
    
    \drawedge(B1,A1){}
    \drawedge(B2,A2){}
    \drawedge(B3,A3){}
    \drawedge(B4,A4){}
    \drawedge(B5,A5){}
    \drawedge(B6,A6){}
    \drawedge(B7,A7){}
    \drawedge(B8,A8){}
    
    \drawedge(C11,B1){}
    \drawedge(C12,B1){}
    \drawedge(C21,B2){}
    \drawedge(C22,B2){}
    \drawedge(C31,B3){}
    \drawedge(C32,B3){}
    \drawedge(C41,B4){}
    \drawedge(C42,B4){}
    \drawedge(C51,B5){}
    \drawedge(C52,B5){}
    \drawedge(C61,B6){}
    \drawedge(C62,B6){}
    \drawedge(C71,B7){}
    \drawedge(C72,B7){}
    \drawedge(C81,B8){}
    \drawedge(C82,B8){}

    \node(D1)(70,-15){$\infty$}
    \node(E1)(70,0){`0'}
    \node(F1)(62,15){13}
    \node(F2)(78,15){39}
    \node(G1)(62,30){0}
    \node(G2)(78,30){26}
    
    \drawedge(E1,D1){}
    \drawedge(F1,E1){}
    \drawedge(F2,E1){}
    \drawedge(G1,F1){}
    \drawedge(G2,F2){}
    
    \drawloop[loopangle=-90](D1){}
\end{picture}
\end{center}

The following connected component is due to the elements of $B_1$.
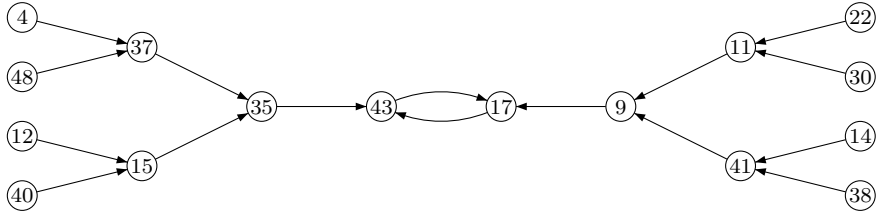
\begin{center}
\begin{picture}(90, 40)(-45,-20)
	\unitlength=2.8pt
    \gasset{Nw=4,Nh=4,Nmr=3,curvedepth=0}
    \thinlines
   \footnotesize
    \node(A1)(8,0){17}
    \node(A2)(-8,0){43}
    
    \node(B1)(24,0){9}
    \node(B2)(-24,0){35}
    
    \node(C11)(40,8){11}
    \node(C12)(40,-8){41}
    \node(C21)(-40,8){37}
    \node(C22)(-40,-8){15}
    
    \node(D111)(56,12){22}
    \node(D112)(56,4){30}
    \node(D121)(56,-4){14}
    \node(D122)(56,-12){38}
    
    \node(D211)(-56,12){4}
    \node(D212)(-56,4){48}
    \node(D221)(-56,-4){12}
    \node(D222)(-56,-12){40}

	\gasset{curvedepth=2}
    \drawedge(A1,A2){}
    \drawedge(A2,A1){}
    
    \gasset{curvedepth=0}
    \drawedge(B1,A1){}
    \drawedge(B2,A2){}
    \drawedge(C11,B1){}
    \drawedge(C12,B1){}
    \drawedge(C21,B2){}
    \drawedge(C22,B2){}
    
    \drawedge(D111,C11){}
    \drawedge(D112,C11){}
    \drawedge(D121,C12){}
    \drawedge(D122,C12){}
    \drawedge(D211,C21){}
    \drawedge(D212,C21){}
    \drawedge(D221,C22){}
    \drawedge(D222,C22){}
\end{picture}
\end{center}

We deal now with the graph $G_{\theta_{38}}^{53}$. Since ${\alpha}^{-2} \equiv 38 \pmod{\pi_{53}}$, we define $\rho_0 = {\alpha}$ and study the iterations of the map $\theta_{38}$ on $A_1$ by means of the iterations of $[\rho_0]$ on $S \cong R / (\pi_{53}-1) R$. We have that
\begin{equation*}
S = R / \rho_0^2 R \times R / \rho_1 R,
\end{equation*}
where $\rho_1 = 1+4i$. The length of the cycles due to the elements of $A_1$ and the depth of the trees attached to $\theta_{38}$-periodic elements can be found as explained in the case of the map $\theta_{15}$.

The iterations of $\theta_{38}$ on $B_1$ can be studied by means of the iterations of $[\rho_0]$ on $S \cong R / (\pi_{53}+1) R$. We have that
\begin{equation*}
S = R / \rho_0^3 R \times R / \rho_1 R,
\end{equation*}
where $\rho_1 = 2+i$. Since $e_0 = 3$, the depth of the trees rooted in elements of $B_1$ is $3$. The $\theta_{38}$-periodic elements in $B_1$ are $x$-coordinates of the points $P=(0,P_1) \in S$, where $P_1 \not = 0$. The $4$ points $P_1 \not = 0$ in $R/\rho_1 R$ have additive order $5$. They give rise to $2$ cycles of length $1$, since $1$ is the smallest among the positive integers $s$ such that $\rho_1$ divides either $\rho_0^s+1$ or $\rho_0^s-1$. 

The graph $G_{\theta_{38}}^{53}$ is here represented. The following $2$ connected components are due to the elements of $A_1$. 
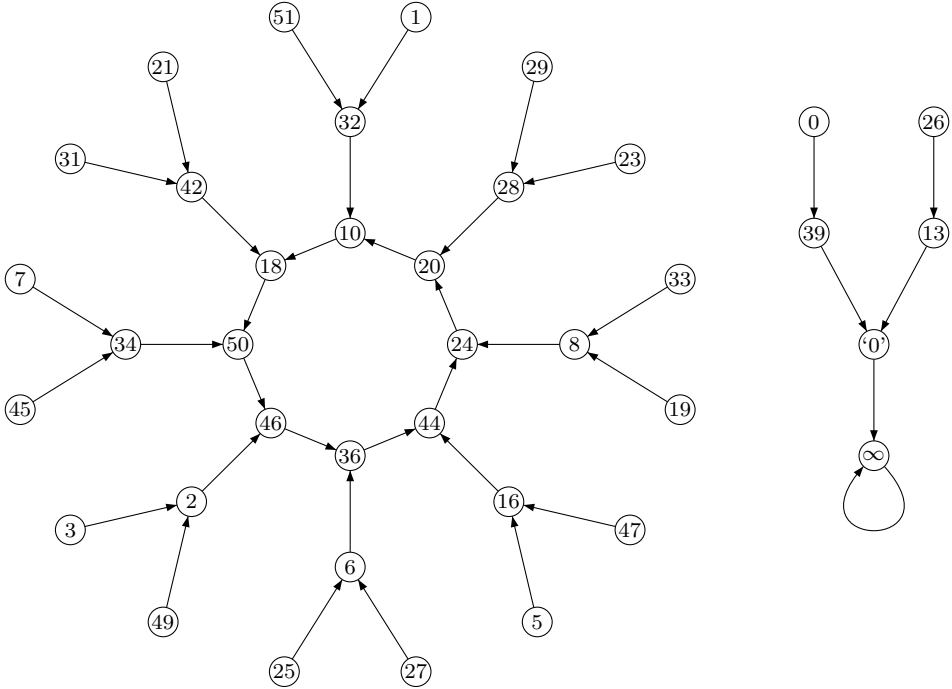
\begin{center}
\begin{picture}(90, 90)(-28,-45)
	\unitlength=2.8pt
    \gasset{Nw=4,Nh=4,Nmr=3,curvedepth=0}
    \thinlines
   	\footnotesize
    \node(A1)(15,0){24}
    \node(A2)(10.6,10.6){20}
    \node(A3)(0,15){10}
    \node(A4)(-10.6,10.6){18}
    \node(A5)(-15,0){50}
    \node(A6)(-10.6,-10.6){46}
    \node(A7)(0,-15){36}
    \node(A8)(10.6,-10.6){44}
    
    \node(B1)(30,0){8}
    \node(B2)(21.2,21.2){28}
    \node(B3)(0,30){32}
    \node(B4)(-21.2,21.2){42}
    \node(B5)(-30,0){34}
    \node(B6)(-21.2,-21.2){2}
    \node(B7)(0,-30){6}
    \node(B8)(21.2,-21.2){16}
    
    \node(C12)(44.1,8.8){33}
    \node(C21)(37.4,25){23}
    \node(C22)(25,37.4){29}
    \node(C31)(8.8,44.1){1}
    \node(C32)(-8.8,44.1){51}
    \node(C41)(-25,37.4){21}
    \node(C42)(-37.4,25){31}
    \node(C51)(-44.1,8.8){7}
    \node(C52)(-44.1,-8.8){45}
    \node(C61)(-37.4,-25){3}
    \node(C62)(-25,-37.4){49}
    \node(C71)(-8.8,-44.1){25}
    \node(C72)(8.8,-44.1){27}
    \node(C81)(25,-37.4){5}
    \node(C82)(37.4,-25){47}
    \node(C11)(44.1,-8.8){19}
 
    \drawedge(A1,A2){}
    \drawedge(A2,A3){}
    \drawedge(A3,A4){}
    \drawedge(A4,A5){}
    \drawedge(A5,A6){}
    \drawedge(A6,A7){}
    \drawedge(A7,A8){}
    \drawedge(A8,A1){}
    
    \drawedge(B1,A1){}
    \drawedge(B2,A2){}
    \drawedge(B3,A3){}
    \drawedge(B4,A4){}
    \drawedge(B5,A5){}
    \drawedge(B6,A6){}
    \drawedge(B7,A7){}
    \drawedge(B8,A8){}
    
    \drawedge(C11,B1){}
    \drawedge(C12,B1){}
    \drawedge(C21,B2){}
    \drawedge(C22,B2){}
    \drawedge(C31,B3){}
    \drawedge(C32,B3){}
    \drawedge(C41,B4){}
    \drawedge(C42,B4){}
    \drawedge(C51,B5){}
    \drawedge(C52,B5){}
    \drawedge(C61,B6){}
    \drawedge(C62,B6){}
    \drawedge(C71,B7){}
    \drawedge(C72,B7){}
    \drawedge(C81,B8){}
    \drawedge(C82,B8){}

    \node(D1)(70,-15){$\infty$}
    \node(E1)(70,0){`0'}
    \node(F1)(62,15){39}
    \node(F2)(78,15){13}
    \node(G1)(62,30){0}
    \node(G2)(78,30){26}
    
    \drawedge(E1,D1){}
    \drawedge(F1,E1){}
    \drawedge(F2,E1){}
    \drawedge(G1,F1){}
    \drawedge(G2,F2){}
    
    \drawloop[loopangle=-90](D1){}
\end{picture}
\end{center}

The following $2$ connected components are due to the elements of $B_1$.
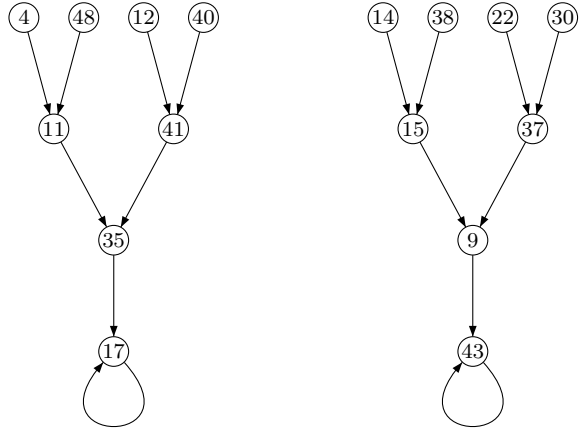
\begin{center}
\begin{picture}(90, 60)(-45,-7)
	\unitlength=2.8pt
    \gasset{Nw=4,Nh=4,Nmr=3,curvedepth=0}
    \thinlines
   \footnotesize
    \node(A1)(24,0){43}
    \node(A2)(-24,0){17}
    
    \node(B1)(24,15){9}
    \node(B2)(-24,15){35}
    
    \node(C11)(16,30){15}
    \node(C12)(32,30){37}
    \node(C21)(-16,30){41}
    \node(C22)(-32,30){11}
    
    \node(D111)(12,45){14}
    \node(D112)(20,45){38}
    \node(D121)(28,45){22}
    \node(D122)(36,45){30}
    
    \node(D211)(-12,45){40}
    \node(D212)(-20,45){12}
    \node(D221)(-28,45){48}
    \node(D222)(-36,45){4}

    \drawedge(B1,A1){}
    \drawedge(B2,A2){}

    \drawedge(C11,B1){}
    \drawedge(C12,B1){}
    \drawedge(C21,B2){}
    \drawedge(C22,B2){}
	   
	\drawedge(D111,C11){}
    \drawedge(D112,C11){}
    \drawedge(D121,C12){}
    \drawedge(D122,C12){}
    \drawedge(D211,C21){}
    \drawedge(D212,C21){}
    \drawedge(D221,C22){}
    \drawedge(D222,C22){}	

	\drawloop[loopangle=-90](A1){}
	\drawloop[loopangle=-90](A2){}
\end{picture}
\end{center}

\end{example}

\section{Case 3: $k$ is a root of $x^2+ \frac{1}{2}x+\frac{1}{2} \in \F_p [x]$ with $p \equiv 1, 2,$ or $4 \pmod{7}$}
The endomorphism ring of the elliptic curve with equation $y^2 = x^3 - 35 x + 98$ over $\Q$ is isomorphic to $\Z \left[ \alpha \right]$, where $\alpha= \frac{1+ \sqrt{-7}}{2}$, as one can deduce from \cite{sil}, Proposition 2.3.1 (iii). In particular, the curve possesses an endomorphism of degree $2$, namely the map $[\alpha]$, which takes a point $(x,y)$ of the curve to
\begin{equation*}
[\alpha] (x,y) = \left(\alpha^{-2} \left(x - \frac{7 (1- \alpha)^4}{x + \alpha^2-2} \right), \alpha^{-3} y \left(1 + \frac{7 (1- \alpha)^4}{(x+ \alpha^2-2)^2} \right)  \right).
\end{equation*}

If $p$ is an odd prime such that
\begin{equation}\label{case3_eq2}
p \equiv 1, 2 \text{ or } 4 \pmod{7},
\end{equation}
then the elliptic curve with equation $y^2 = x^3 - 35 x + 98$ has good reduction modulo $p$ (see \cite{sil_a}, Chapter V, Proposition 5.1 or \cite{mil}, page 59). Hence, from now on we suppose that $p$ is a fixed prime number as in (\ref{case3_eq2}) and we denote by $E$ the elliptic curve with equation 
\begin{equation*}
y^2 = x^3 - 35 x + 98
\end{equation*}
over $\F_p$.
Being $-7$ a quadratic residue in $\F_p$ we also get that
\begin{equation}\label{case3_eq1}
x^2-x+2 = 0
\end{equation}
has two solutions $\omega, \overline{\omega}$ in $\F_p$.

Fixed a positive integer $n$ we want to study the iterations over $\Pro (\F_{p^n})$ of the maps $\theta_{k_{\sigma}}$, for $\sigma \in \{\omega, \overline{\omega} \}$, where
\begin{equation}
k_{\omega} \equiv \frac{\omega-1}{2} \pmod{p}, \quad k_{\overline{\omega}} \equiv \frac{\overline{\omega}-1}{2} \pmod{p}.
\end{equation}

Firstly we notice that $k_{\omega}$ and $k_{\overline{\omega}}$ are the two roots in $\F_p$ of the quadratic equation
\begin{equation}
x^2 + \frac{1}{2} x + \frac{1}{2} = 0.
\end{equation}

Now we show that, for $\sigma \in \{\omega, \overline{\omega} \}$, the map $\theta_{k_{\sigma}}$ is conjugated to the  map $\eta_{k_{\sigma}}$ defined over $\Pro(\F_{p^n})$ as follows: 
\begin{displaymath}
\eta_{k_{\sigma}} (x) = 
\begin{cases}
\infty & \text{if $x = - \sigma^2+2$ or $\infty$,}\\
\frac{1}{{\sigma}^2} \cdot \left( x - \frac{7 \cdot (1-\sigma)^4}{x+\sigma^2-2} \right) & \text{otherwise}.
\end{cases}
\end{displaymath} 
Before proving this fact, we notice that the two maps $\eta_{k_{\sigma}}$, for $\sigma \in \{\omega, \overline{\omega} \}$, are involved in the definition of two endomorphisms of the curve $E$, namely the maps $e_{k_{\sigma}}$ which take a rational point $(x,y)$ in $E (\F_{p^n})$ to
\begin{eqnarray*}
e_{k_{\sigma}} (x,y) = \left(\eta_{k_{\sigma}}(x) , \frac{1}{\sigma^3} \cdot y \cdot \left(1 + \frac{7 \cdot (1- \sigma)^4}{(x+ \sigma^2-2)^2} \right)  \right).
\end{eqnarray*} 
The fact that any map $\theta_{k_{\sigma}}$ is conjugated to the respective map $\eta_{k_{\sigma}}$ enables us to study the iterations of the former maps relying upon the action of the endomorphisms $e_{k_{\sigma}}$ on the rational points of the curve $E$. 

With the aim to prove the conjugation between the maps $\theta_{k_{\sigma}}$ and $\eta_{k_{\sigma}}$, for $\sigma$ equal to $\omega$ or $\overline{\omega}$ respectively, we set
\begin{equation*}
b_{\sigma} \equiv k_{\sigma} + \frac{5}{2} \pmod{p}, \quad c_{\sigma} \equiv - \frac{1}{4 k_{\sigma} +1} \pmod{p}, \quad d_{\sigma} \equiv \frac{\frac{1}{2}+k_{\sigma}}{4k_{\sigma}+1} \pmod{p}
\end{equation*}
and define two bijective maps $\chi_{k_{\sigma}}$ on $\Pro (\F_{p^n})$ in such a way:
\begin{displaymath}
\chi_{k_{\sigma}} : 
\begin{array}{lll}
x & \mapsto & 
\begin{cases}
\frac{1}{c_{\sigma}} & \text{if $x = \infty$},\\
\infty & \text{if $x = - \frac{d_{\sigma}}{c_{\sigma}}$},\\
\frac{x+b_{\sigma}}{c_{\sigma} x+d_{\sigma}} & \text{in all other cases}.
\end{cases}
\end{array}
\end{displaymath}
The inverses of the maps $\chi_{k_{\sigma}}$ are
\begin{displaymath}
\chi_{k_{\sigma}}^{-1} : 
\begin{array}{lll}
x & \mapsto & 
\begin{cases}
\infty & \text{if $x = \frac{1}{c_{\sigma}}$},\\
- \frac{d_{\sigma}}{c_{\sigma}} & \text{if $x = \infty$},\\
\frac{d_{\sigma} x-b_{\sigma}}{-c_{\sigma} x+1} & \text{in all other cases}.
\end{cases}
\end{array}
\end{displaymath}

With the notation just introduced we can prove the following technical result.
\begin{lemma}\label{case3_lem_2}
For any $x \in \Pro(\F_{p^n})$,
\begin{equation}
\chi_{k_{\sigma}}^{-1} \circ \eta_{k_{\sigma}} \circ \chi_{k_{\sigma}} (x) = \theta_{k_{\sigma}} (x). 
\end{equation}
\end{lemma}

\begin{proof}
As a first step we notice that $\sigma^2 \equiv 2 k_{\sigma} -1 \pmod{p}$. Therefore, 
\begin{equation*}
\eta_{k_{\sigma}} (x) = \frac{1}{2 k_{\sigma}-1} \cdot \left(x - \frac{14 \cdot (3 k_{\sigma}+1)}{x+ 2 k_{\sigma} - 3} \right).
\end{equation*}
Taking into account the fact that $k_{\sigma}^2 \equiv -\frac{1}{2} k_{\sigma} - \frac{1}{2} \pmod{p}$, we have that in $\F_p (x)$
\begin{equation*}
\eta_{k_{\sigma}} \circ \chi_{k_{\sigma}} (x) = \frac{(56 k_{\sigma}+28) \cdot x^2 + (-84 k_{\sigma} + 98) \cdot x + (56 k_{\sigma} +28)}{(8 k_{\sigma} - 12) \cdot x^2 + (-20 k_{\sigma} +2) \cdot x + (8 k_{\sigma} -12)}.
\end{equation*}
Since in $\F_p (x)$
\begin{equation*}
\chi_{k_{\sigma}}^{-1} \circ \eta_{k_{\sigma}} (x) = \frac{(42k_{\sigma}+7) \cdot x^2 + (42 k_{\sigma}+7)}{(-14 k_{\sigma} + 35) x} \equiv k_{\sigma} \cdot \frac{x^2+1}{x},
\end{equation*}
we get the thesis.
\end{proof}

According to Lemma \ref{case3_lem_2} any of the two maps $\theta_{k_{\sigma}}$ is conjugated to the respective map $\eta_{k_{\sigma}}$. Therefore, since the graphs $G_{\theta_{k_{\sigma}}}^{p^n}$ are isomorphic to the graphs $G_{\eta_{k_{\sigma}}}^{p^n}$, we will concentrate on the study of these latter graphs.

\subsection{Structure of the graphs $G_{\eta_{k_{\sigma}}}^{p^n}$}
We remind that the endomorphism ring of the curve $y^2 = x^3 - 35 x + 98$ over $\Q$ is isomorphic to $\Z \left[ \alpha \right]$. Therefore, according to \cite{lan}, Chapter 13, Theorem 12, the endomorphism ring $\End(E)$ of $E$ over $\F_p$ is isomorphic to $R = \Z \left[ \alpha \right]$ too. We notice that $R$ is an Euclidean ring with Euclidean function
\begin{equation*}
N(a + b \alpha) = (a+ b \alpha) \cdot \overline{(a + b \alpha)},
\end{equation*} 
for any arbitrary choice of $a, b$ in $\Z$ (here $\overline{(\cdot)}$ denotes the complex conjugation).

By \cite{wit}, Theorem 2.3(a), there exists an isomorphism
\begin{equation*}
\psi_n : E(\F_{p^n}) \to R/(\pi_p^n-1) R,
\end{equation*} 
where $\pi_p$ is the Frobenius endomorphism. We remind that, by \cite{wit}, Theorem 2.4, the representation of $\pi_p$ as an element of $R$ is 
\begin{equation*}
\pi_p = \frac{p+1-m+\sqrt{d}}{2},
\end{equation*}
where
\begin{eqnarray*}
m & = & |E(\F_p)|,\\
d & = & (p+1-m)^2-4p. 
\end{eqnarray*}
Moreover,
\begin{equation*}
2 = e_{k_{\omega}} \circ e_{k_{\overline{\omega}}},
\end{equation*}
being $2$ the duplication map over the curve $E$. Since in $R$
\begin{equation*}
2 = \alpha \cdot \overline{\alpha},
\end{equation*}
we get that the endomorphisms $e_{k_{\omega}}$ and $e_{k_{\overline{\omega}}}$ are represented in $R$ by $\alpha$ and $\overline{\alpha}$. 

Let $\sigma \in \left\{ \omega, \overline{\omega} \right\}$ and fix once for the remaining part of the current section $\sigma= \omega$ or $\sigma = \overline{\omega}$.
Now we partition the elements of $\Pro(\F_{p^n})$ in two subsets:
\begin{eqnarray*}
A_n & = & \{x \in \F_{p^n} : (x, y) \in E(\F_{p^n}) \text{ for some $y \in \F_{p^n}$} \} \cup \{ \infty \},\\
B_n & = & \{x \in \F_{p^n} : (x, y) \in E(\F_{p^{2n}}) \text{ for some $y \in \F_{p^{2n}} \backslash \F_{p^n}$} \}.
\end{eqnarray*} 

As a first step in our study we prove some  technical results.
\begin{lemma}\label{case3_lem_4}
Let $\tilde{x} \in \F_{p^n}$. Then, in $E(\F_{p^{2n}})$ there are two distinct rational points, $(\tilde{x}, \tilde{y})$ and $(\tilde{x}, -\tilde{y})$, with such an $x$-coordinate except for 
\begin{equation*}
\tilde{x} \in \{-7, {\sigma}+3, {\overline{\sigma}}+3 \},
\end{equation*}
in which case $\tilde{y} = 0$.
\end{lemma}
\begin{proof}
The equation $y^2 = \tilde{x}^3-35 \tilde{x}+98$ has exactly two distinct roots $y_1$ and $y_2$ in $\F_{p^{2n}}$ except in the case that $\tilde{x}$ makes vanish the polynomial $x^3-35x+98$, whose set of roots is $\left\{-7, {\sigma}+3, {\overline{\sigma}}+3 \right\}$.
\end{proof}

\begin{lemma}\label{case3_lem_3}
The following hold.
\begin{itemize}
\item $- \sigma^2 + 2 = \overline{\sigma}+3$ in $\F_p$;
\item $\eta_{k_{\sigma}} (-7) = {\sigma}+3$;
\item $\eta_{k_{\sigma}} ({\overline{\sigma}}+3) = \infty$;
\item $\eta_{k_{\sigma}} (x) = x$ if and only if $x \in \{{\sigma}+3, -2{\sigma}+1, \infty \}$.  
\end{itemize}
\end{lemma}
\begin{proof}
All assertions can be checked by direct computation reminding that 
\begin{equation*}
{\sigma}^2 - {\sigma} + 2 = 0 \quad \text{ and } \quad {\sigma} + {\overline{\sigma}} = 1.
\end{equation*}
Let us begin with the first assertion:
\begin{equation*}
- \sigma^2 + 2 = - \sigma + 4 = \overline{\sigma} + 3.
\end{equation*}

The second assertion can be proved as follows:
\begin{eqnarray*}
\eta_{k_{\sigma}} (-7) & = & \frac{1}{{\sigma}^{2}} \cdot \left( -7 - \frac{7 \cdot (1-{\sigma})^4}{-7+{\sigma}^2-2} \right) = \left( - \frac{1}{4} {\sigma} - \frac{1}{4} \right) \cdot \left(-7 - \frac{21 {\sigma}-7}{{\sigma}-11} \right) \\
& = & \left( - \frac{1}{4} {\sigma} - \frac{1}{4} \right) \cdot \left(-7 - (21 {\sigma} - 7) \cdot \left( - \frac{1}{112} {\sigma} - \frac{5}{56} \right) \right) = {\sigma}+3. 
\end{eqnarray*}

The third assertion follows from the first one and from the definition of $\eta_{k_{\sigma}}$.

Finally, $\eta_{k_{\sigma}} (x) = x$ if and only if $x = \infty$ or 
\begin{equation*}
{\sigma}^2 \cdot (x+{\sigma}^2-2) \cdot \eta_{k_{\sigma}} (x) = {\sigma}^2 \cdot (x+{\sigma}^2-2) \cdot x.
\end{equation*}
After some algebraic manipulations we get that this latter is equivalent to the quadratic equation
\begin{equation*}
(-{\sigma} + 3) \cdot x^2 + (6 {\sigma} - 10) \cdot x - (21 {\sigma} - 7) = 0,
\end{equation*}
whose roots are ${\sigma}+3$ and $-2 {\sigma} +1$.
\end{proof}

\begin{lemma}
The map $\eta_{k_{\sigma}}$ takes the elements of $A_n$ to $A_n$ and the elements of $B_n$ to $B_n$.
\end{lemma}
\begin{proof}
By definition $\eta_{k_{\sigma}} (\infty) = \infty \in A_n$ and, by Lemma \ref{case3_lem_3}, we have that $- \sigma^2+2= \overline{\sigma}+3 \in A_n$, since $(\overline{\sigma}+3, 0) \in E(\F_{p^n})$ according to Lemma \ref{case3_lem_4}.

If $\tilde{x} \in A_n \backslash \{- \sigma^2+2, \infty \}$, then there exists $\tilde{y} \in \F_{p^n}$ such that $(\tilde{x}, \tilde{y}) \in E(\F_{p^{n}})$. Therefore, $e_{k_{\sigma}} (\tilde{x}, \tilde{y}) \in E(\F_{p^n})$ and $\eta_{k_{\sigma}} (\tilde{x}) \in A_n$.

Consider finally an element $\tilde{x} \in B_n$. In this case $(\tilde{x}, \tilde{y}) \in E(\F_{p^{2n}})$ for some $\tilde{y} \in \F_{p^{2n}} \backslash \F_{p^n}$. Moreover, $(\eta_{k_{\sigma}}(\tilde{x}), \tilde{y}_1)$ and $(\eta_{k_{\sigma}}(\tilde{x}), \tilde{y}_2)$, with
\begin{eqnarray*}
\tilde{y}_j & = & (-1)^j \cdot \frac{1}{{\sigma}^3} \cdot \tilde{y} \cdot \left(1 + \frac{7 \cdot (1- {\sigma})^4}{(x+ {\sigma}^2-2)^2}\right), \quad \text{for $j \in \{1, 2 \}$,}  
\end{eqnarray*}  
are the only rational points in $E(\F_{p^{2n}})$ having $\tilde{x}$ as $x$-coordinate. In addition, $\tilde{y} \not \in \F_{p^n}$, unless $\tilde{y} = 0$. Nevertheless, this happens if and only if $\tilde{x} \in \{-7, \sigma+3, \overline{\sigma}+3 \} \subseteq A_n$. Since $\tilde{x} \not \in A_n$ we get the thesis.
\end{proof}

Before dealing with the iterations of $\eta_{k_{\sigma}}$ on $B_n$, we introduce  the  set
\begin{eqnarray*}
E(\F_{p^{2n}})_{B_n} = \left\{(x,y) \in E(\F_{p^{2n}}) : x \in B_n \right\}.
\end{eqnarray*}

Moreover we define the set
\begin{equation*}
E(\F_{p^{2n}})^*_{B_n} = \{O, (-7,0), (\sigma+3, 0), (\overline{\sigma}+3, 0) \},
\end{equation*}
where $O$ denotes the point at infinity of $E$.

The following holds.
\begin{lemma}\label{case3_lem_1}
Let $\tilde{x} \in \F_{p^{2n}}$ and $P = (\tilde{x}, \tilde{y}) \in E(\F_{p^{2n}})$. We have that $(\pi_p^n+1) P = O$ if and only if $P \in E(\F_{p^{2n}})_{B_n} \cup E(\F_{p^{2n}})_{B_n}^*$.
\end{lemma}
\begin{proof}
The proof is verbatim the same as in Lemma \ref{case2_lem_2}, just considering the new definition of the set $E(\F_{p^{2n}})_{B_n}^*$.
\end{proof}

With the notation till now introduced and in virtue of Lemma \ref{case3_lem_1} we can say that there exists an isomorphism
\begin{equation*}
\widetilde{\psi}_n : E(\F_{p^{2n}})_{B_n} \cup E(\F_{p^{2n}})^*_{B_n} \to R / (\pi_p^n+1) R.
\end{equation*}

All considered we can study the graph $G_{\eta_{k_{\sigma}}}^{p^n}$ separately on the elements of $A_n$ and $B_n$. To do that we will rely upon the action of $[\rho_0]$ on the elements of $R/(\pi_p^n-1)R$ (resp. $R/(\pi_p^n+1)R$), where

\begin{displaymath}
\rho_0 = 
\begin{cases}
\alpha, & \text{if $\alpha \equiv \sigma \pmod{\pi_p}$;}\\
\overline{\alpha}, & \text{if $\overline{\alpha} \equiv \sigma \pmod{\pi_p}$.}
\end{cases}
\end{displaymath}

Suppose that  $\pi_p^n-1$ (resp. $\pi_p^n+1$) factors in $R$, up to units, as
\begin{equation}\label{case3_eq_3}
\rho_0^{e_{0}} \cdot \left( \prod_{i=1}^v p_i^{e_i} \right) \cdot \left( \prod_{i = v+1}^w {r_i}^{e_i} \right) \cdot (\sqrt{-7})^{e_f},
\end{equation}
where 
\begin{enumerate}
\item all $e_i$, for $0 \leq i \leq w$, and $e_f$ are non-negative integers;
\item $N(\rho_0^{e_{0}}) = 2^{e_{0}}$;
\item  for $1 \leq i \leq v$ the elements $p_i \in \Z$ are distinct primes of $R$ and $N( p_i^{e_i} ) = p_i^{2 e_i}$;
\item for $v+1 \leq i \leq w$ the elements $r_i \in R \backslash \Z$  are distinct primes of $R$, different from $\rho_0$ and $\sqrt{-7}$,  and $N( r_i^{e_i} ) = p_i^{e_i}$, for some rational integer $p_i$ such that $r_i \overline{r}_i = p_i$.
\end{enumerate}

For the sake of clarity we introduce the index set 
\begin{equation*}
J = \{i : 0 \leq i \leq w \} \cup \{ f \}
\end{equation*}
and define 
\begin{displaymath}
\rho_j = 
\begin{cases}
p_j & \text{if $j \in \{1, \dots, v \}$},\\
r_j & \text{if $j \in \{v+1, \dots, w \}$},\\
\sqrt{-7} & \text{if $j=f$.}
\end{cases}
\end{displaymath}
As a consequence of the factorization (\ref{case3_eq_3}) the ring $R / (\pi_p^n - 1) R$ (resp. $R / (\pi_p^n + 1) R$) is isomorphic to 
\begin{equation}\label{orb4}
S = \prod_{j \in J} R / \rho_j^{e_j} R.
\end{equation}

The additive structure of the quotient rings involved in (\ref{orb4}) has been discussed in Section \ref{section_case_2}. Herein we just notice that
the additive group of $R / (\sqrt{-7})^{e_f} R$ is isomorphic to the direct sum of two cyclic groups of order $7^{e_f/2}$, if $e_f$ is even, or to the direct sum of two cyclic groups of order respectively $7^{(e_f-1)/2}$ and $7^{(e_f+1)/2}$, if $e_f$ is odd. In the case that $e_f$ is even, for each integer $0 \leq h_f \leq e_f/2$ there are $N_{h_f}$ elements in $R/({\sqrt{-7}})^{e_f} R$ of order $7^{h_f}$, where  
\begin{displaymath}
N_{h_f} = \left\{
\begin{array}{ll}
1 & \text{ if $h_f = 0$,}\\
7^{2 h_f} - 7^{2(h_{f}-1)} & \text{ otherwise.}
\end{array}
\right.
\end{displaymath}
If, on the contrary, $e_f$ is odd, then  
\begin{displaymath}
N_{h_f} = \left\{
\begin{array}{ll}
1 & \text{ if $h_f = 0$,}\\
7^{2 h_f} - 7^{2(h_{f}-1)} & \text{ if $1 \leq h_f \leq (e_f-1)/2$,}\\
7^{e_f} - 7^{e_f-1} & \text{ if $h_f = (e_f+1)/2$.}
\end{array}
\right.
\end{displaymath}

If $(x,y)$ is a rational point of $E(\F_{p^n})$ (resp. $E(\F_{p^{2n}})_{B_n}$), then we write $P_{(x,y)}$ for the image of $(x,y)$ in $S$.

Now we define the sets 
\begin{eqnarray*}
Z_{j} & = & \{0, 1, \dots, e_j \}, \quad \text{for any $j \in J \backslash \{ f \}$}, \\
Z_f & = & \left\{0, 1, \dots, \lceil e_f/2 \rceil \right\}
\end{eqnarray*}
and
\begin{equation*}
H = \prod_{j \in J} Z_{j}.
\end{equation*} 

If $P = (P_{0}, P_1, \dots, P_f) \in S$, then we define $h^P = (h_{0}^P, h_{1}^P, \dots, h_f^P)$ in $H$ if 
\begin{itemize}
\item $P_{0}$ has additive order $2^{h_{0}^P}$ in $R / \rho_0^{e_0} R$;
\item each $P_i$, for $1 \leq i \leq w$, has additive order $p_i^{h_i^P}$ in $R / \rho_i^{e_i} R$;
\item $P_f$ has additive order $7^{h_f^P}$ in $R / \rho_f^{e_f} R$.
\end{itemize}

Moreover, we define $o(P) = \left(o(P_{0}), o(P_{1}), \dots, o(P_f) \right)$, where $o(P_j)$ denotes, for any $j \in J$,  the additive order of $P_j$ in $R / \rho_j^{e_j} R$.

The following two lemmas furnish a characterization of $\eta_{k_{\sigma}}$-periodic elements.

\begin{lemma}\label{case3_lem_7}
Let $\tilde{x} \in A_n$ (resp. $B_n$) be $\eta_{k_{\sigma}}$-periodic. Then, one of the following holds:
\begin{itemize}
\item $\tilde{x} = \infty$;
\item $\tilde{x}$ is the $x$-coordinate of a rational point $(\tilde{x},\tilde{y}) \in E(\F_{p^n})$ (resp. $E(\F_{p^{2n}})_{B_n}$) such that $P_{(\tilde{x},\tilde{y})} = (P_{0}, P_{1}, \dots, P_f)$, where $P_{0} = 0$.
\end{itemize}

\end{lemma}
\begin{proof}
The proof is verbatim the same as in Lemma \ref{case2_lem_7}.
\end{proof}

\begin{lemma}\label{case3_lem_5}
Let $P = (P_{0}, P_{1}, \dots, P_f)$ be a point in $S$ such that $P_{0} = 0$ and denote  by $l_j$, for any $j \in J$, the smallest among the positive integers $s$ such that either $[\rho_0]^s \cdot P_{j} = P_{j}$ or $[\rho_0]^s \cdot P_j = -P_{j}$.

Let $l' = \lcm (l_{0}, l_{1}, \dots, l_f)$ and denote by $l$ the smallest among the positive integers $s$ such that $[\rho_0]^s P = P$ or $[\rho_0]^s P = -P$.

Then, any $l_{j}$ is the smallest among the positive integers $s$ such that $\rho_j^{h_{j}^P}$ divides either $\rho_0^{s} + 1$ or  $\rho_0^{s} - 1$ in $R$. Moreover, $l$ is determined as follows: 
\begin{displaymath}
l = 
\begin{cases}
l' & \text{if either $\rho_j^{h_{j}^P} \mid (\rho_0^{l'} +1) $ for all $j \in J$ or $\rho_j^{h_{j}^P} \mid (\rho_0^{l'} - 1) $ for all $j \in J$},\\
2 l' & \text{otherwise}.
\end{cases}
\end{displaymath}
\end{lemma}
\begin{proof}
Fix a $j \in J$. The proof of the current lemma is the same as in Lemma \ref{case2_lem_5}, replacing all occurrences of the index $i$ by $j$.
\end{proof}

The following holds.
\begin{lemma}\label{case3_lem_6}
Let $x_0 \in \F_{p^n}$ be $\eta_{k_{\sigma}}$-periodic, $s$ a positive integer and $x_s = \eta_{k_{\sigma}}^s (x_0)$. Let $(x_0,y_0) \in E(\F_{p^n})$ (resp. $E(\F_{p^{2n}})_{B_n}$) for some $y_0 \in \F_{p^n}$ (resp. $\F_{p^{2n}}$) and $(x_s, y_s) \in E(\F_{p^n})$ (resp. $E(\F_{p^{2n}})_{B_n}$) for some $y_s \in \F_{p^n}$ (resp. $\F_{p^{2n}}$). If $Q^{(0)} = P_{(x_0,y_0)}$ and $Q^{(s)} = P_{(x_s,y_s)}$, then $h^{Q^{(0)}} = h^{Q^{(s)}}$.
\end{lemma}
\begin{proof}
The proof is the same as in Lemma \ref{case2_lem_6}, defining $Q^{(0)} = (Q_{0}^{(0)}, Q_{1}^{(0)}, \dots, Q_f^{(0)})$, $Q^{(s)} = \pm ([\rho_0]^s Q_{0}^{(0)}, \dots, [\rho_0]^s Q_f^{(0)})$ and replacing $i \not = 0$ with $j \in J \backslash \{ 0 \}$.
\end{proof}

As in Section \ref{section_case_2} we introduce the following notation.

\begin{definition}
If $x \in \Pro (\F_{p^n})$ is $\eta_{k_{\sigma}}$-periodic, then we denote by
\begin{equation*}
\left< x \right>_{\eta_{k_{\sigma}}}  = \{ \eta_{k_{\sigma}}^r (x) : r \in \N \} 
\end{equation*}
the elements of the cycle of $x$ with respect to the map $\eta_{k_{\sigma}}$.
\end{definition}

In virtue of Lemmas \ref{case3_lem_5} and \ref{case3_lem_6} we can give the following definition.
\begin{definition}
If $h = (h_{0}, h_{1}, \dots, h_f) \in H$, then $C_h$ denotes the set of all cycles $\left< x \right>_{\eta_{k_{\sigma}}}$ of $G^{p^n}_{\eta_{k_{\sigma}}}$, where $x \in A_n$ (resp. $B_n$) and exactly one of the following conditions holds:
\begin{itemize}
\item $x = \infty$ (and $h=(0,0, \dots, 0)$);
\item $(x, y) \in E(\F_{p^n})$ (resp. $E (\F_{p^{2n}})_{B_n}$) for some $y \in \F_{p^n}$ (resp. $\F_{p^{2n}}$) and $h^{P_{(x,y)}} = h$. 
\end{itemize}
Moreover we introduce the following notations:
\begin{itemize}
\item $l_h$ is the length of the cycles in $C_h$;
\item $C_{A_n} = \{ \langle x \rangle_{\eta_{k_{\sigma}}} \in  G^{p^n}_{\eta_{k_{\sigma}}} : x \in A_n \}$;
\item $C_{B_n} = \{ \langle x \rangle_{\eta_{k_{\sigma}}} \in  G^{p^n}_{\eta_{k_{\sigma}}} : x \in B_n \}$;
\item $N_h = \dfrac{1}{2 l_h} \cdot \phi(2^{h_{0}}) \cdot \left( \displaystyle\prod_{i=1}^{v} N_{h_i} \right)  \cdot \left( \displaystyle\prod_{i=v+1}^{w} \phi(p_i^{h_i}) \right) \cdot N_{h_f}$.
\end{itemize}
\end{definition}

Now we state two theorems, concerning respectively the cycles of $C_{A_n}$ and $C_{B_n}$.
\begin{theorem}\label{case3_thm_1_A}
Let $S \cong R / (\pi_p^n-1) R$ and let $H_A \subseteq H$ be the set formed by all $h \in H$ such that $h_0 = 0$. Then, 
\begin{equation*}
C_{A_n} = \bigsqcup_{h \in H_A} C_h.
\end{equation*}
Moreover, $|C_h| = 1$ in the following cases:
\begin{itemize}
\item $h \in H_A$ and $h_j=0$ for all $j \in J$;
\item $h \in H_A$, $\rho_{\tilde{j}} = \overline{\rho}_0$ for some $1 \leq \tilde{j} \leq w$, $h_{\tilde{j}} = 1$ and $h_j = 0$ for $j \not = \tilde{j}$.
\end{itemize}
In all other cases $|C_h| = N_h$. 
\end{theorem}

\begin{proof}
Since $C_h \subseteq C_{A_n}$ for any $h \in H_A$, we have that $\displaystyle\bigsqcup_{h \in H_A} C_h \subseteq C_{A_n}$. Vice versa, if $\langle \tilde{x} \rangle_{\eta_{k_{\sigma}}} \in C_{A_n}$, then either $\tilde{x} = \infty$, or $(\tilde{x}, \tilde{y}) \in E(\F_{p^n})$ for some $\tilde{y} \in \F_{p^n}$ and $P = P_{(\tilde{x}, \tilde{y})} = (0, P_1, \dots, P_f)$ by Lemma \ref{case3_lem_7}. Hence, $h_0^P = 0$ and $h^P \in H_A$. Therefore, $ C_{A_n}  \subseteq \displaystyle\bigsqcup_{h \in H_A} C_h$.

We now evaluate $|C_h|$, analysing separately the different cases.
\begin{itemize}
\item Let $\langle \tilde{x} \rangle_{\eta_{k_{\sigma}}} \in C_h$, where $h \in H_A$ and $h_j = 0$ for all $j \in J$.  Then $h=h^P$, where $P_j = 0$ for all $j \in J$. Therefore, $ \tilde{x} = \infty$ and $|C_h| = 1$. 
\item Let $\langle \tilde{x} \rangle_{\eta_{k_{\sigma}}} \in C_h$, where $h \in H_A, h_{\tilde{j}} = 1$ for some  $1 \leq \tilde{j} \leq w$ and $h_j = 0$ for $j \not = \tilde{j}$. Then $h = h^P$, where $P_j = 0$ for $j \not = \tilde{j}$. Moreover, $P_{\tilde{j}}$ has additive order $2$. Hence, also $P$ has additive order $2$ and $P = P_{(\tilde{x}, 0)}$. Therefore, $\tilde{x} \in \{-7, \sigma+3, \overline{\sigma}+3 \}$. Since $-7$ and $\overline{\sigma}+3$ are not $\eta_{k_{\sigma}}$-periodic, we conclude that $\tilde{x} = \sigma +3$ and $|C_h| = 1$.
\item Let $\langle \tilde{x} \rangle_{\eta_{k_{\sigma}}} \in C_h$ for some $h \in H_A$ such that none of the previous conditions occurs. Then, $(\tilde{x},\tilde{y}) \in E(\F_{p^n})$ for some $\tilde{y} \in \F_{p^n}$ and the additive order of $P_{(\tilde{x}, \tilde{y})}$ is not $2$. Hence, there are exactly two distinct rational points in $E(\F_{p^n})$ having such a $x$-coordinate. Since the length of the cycle $\langle \tilde{x} \rangle_{\eta_{k_{\sigma}}}$ is $l_h$ and the number of points $Q$ in $S$ such that $h^Q = h$ is 
\begin{equation*}
\phi(2^{h_0}) \cdot \left( \displaystyle\prod_{i=1}^{v} N_{h_i} \right) \cdot \left( \displaystyle\prod_{i=v+1}^{w} \phi(p_i^{h_i}) \right) \cdot N_{h_f},
\end{equation*}
the thesis follows.
\end{itemize}
\end{proof}

\begin{theorem}\label{case3_thm_1_B}
Let $S \cong R / (\pi_p^n+1) R$ and let $H_B \subseteq H$ be the set formed by all $h \in H$ such that
\begin{itemize}
\item $h_{0} = 0$;
\item $h \not = (0, 0, \dots, 0)$;
\item if $\rho_{\tilde{j}} = \overline{\rho}_0$ for some $1 \leq \tilde{j} \leq w$ and $h_{\tilde{j}} = 1$, then $h_j \not = 0$ for some $j \not = \tilde{j}$.
\end{itemize}
Then,
\begin{equation*}
C_{B_n} = \bigsqcup_{h \in H_B} C_h
\end{equation*}
and $|C_h| = N_h$ for any $h \in H_B$.
\end{theorem}

\begin{proof}
Since $C_h \subseteq C_{B_n}$ for any $h \in H_B$, we have that $\displaystyle\bigsqcup_{h \in H_B} C_h \subseteq C_{B_n}$. Vice versa, if $\langle \tilde{x} \rangle_{\eta_{k_{\sigma}}} \in C_{B_n}$, then $\tilde{x} \in \F_{p^n}$ and $(\tilde{x}, \tilde{y}) \in E(\F_{p^{2n}})_{B_n}$ for some $\tilde{y} \in \F_{p^{2n}} \backslash \F_{p^n}$. In particular, $P = P_{(\tilde{x}, \tilde{y})}$ is not the point at infinity and $h^P \not = (0, 0, \dots, 0)$. At the same time $h_0^P = 0$ by Lemma \ref{case3_lem_7}. Moreover, if $\rho_{\tilde{j}} = \overline{\rho}_0$ for some $1 \leq \tilde{j} \leq w$ and $h_{\tilde{j}} =1$, then $h_j \not = 0$ for some $j \not = \tilde{j}$. On the contrary, $P$ would have additive order $2$, namely $\tilde{y} = 0$. Hence, $ C_{B_n}  \subseteq \displaystyle\bigsqcup_{h \in H_B} C_h$.

Let $\langle \tilde{x} \rangle_{\eta_{k_{\sigma}}} \in C_h$ for some $h \in H_B$. Then, $(\tilde{x}, \tilde{y}) \in E(\F_{p^{2n}})_{B_n}$ for some $\tilde{y} \in \F_{p^{2n}}$. Moreover, by definition of $H_B$, the additive order of $P_{(\tilde{x}, \tilde{y})}$ is not $2$. This latter fact implies that there are exactly two distinct points in $E(\F_{p^{2n}})_{B_n}$ having $\tilde{x}$ as $x$-coordinate. Since the length of the cycle $\langle \tilde{x} \rangle_{\eta_{k_{\sigma}}}$ is $l_h$ and the number of points $Q$  in $S$ such that $h^Q = h$ is 
\begin{equation*}
\phi(2^{h_{0}})  \cdot \left( \displaystyle\prod_{i=1}^{v} N_{h_i} \right) \cdot \left( \displaystyle\prod_{i=v+1}^{w} \phi(p_i^{h_i}) \right) \cdot N_{h_f},
\end{equation*}
the thesis follows.
\end{proof} 

In the following we will denote by $V_{A_n}$ (resp. $V_{B_n}$) the set of the $\eta_{k_{\sigma}}$-periodic elements of $A_n$ (resp. $B_n$).  Before proceeding with the description of the trees rooted in vertices of $V_{A_n}$ and $V_{B_n}$ we notice that,  according to \cite{gil},
\begin{eqnarray*}
R / \rho_0^{e_{0}} R & = & \left\{ \sum_{i=0}^{e_{0}-1} \delta_i \cdot [\rho_0]^i: \delta_i = 0 \text{ or } 1 \right\}.
\end{eqnarray*}

The following theorem characterizes the reversed trees having root in $V_{A_n}$ (resp. $V_{B_n}$).
\begin{theorem}\label{case3_thm_2}
Any element $x_0 \in V_{A_n}$ (resp. $V_{B_n}$) is the root of a reversed binary tree of depth $e_{0}$ with the following properties.
\begin{itemize}
\item If $x_0 \not \in \{ \sigma + 3, \infty \}$, then $x_0$ has exactly one child, while all other vertices have two distinct children.
\item If $x_0 \in \{ \sigma + 3, \infty \}$, then $x_0$ and the vertex at the level $1$ of the tree have exactly one child each, while all other vertices have two distinct children. 
\end{itemize}
\end{theorem}
\begin{proof}
The proof that the depth of the tree is $e_0$ and that any vertex at the level $r < e_0$ of the tree has at least one child follows the same lines as the proof of Theorem \ref{case2_thm_2}. More precisely, we have to replace the map $\theta_{k_{\sigma}}$ with the map $\eta_{k_{\sigma}}$, the index set $\{0, 1, \dots, w \}$ with the current index set $J$ and modify consequently the points $P, Q$ and $\tilde{Q}$ considering the new index set.

As regards the number of children of any vertex $x_r \not = \infty$ at a certain level $r < e_0$ of the tree rooted in $x_0$, we now prove that $x_r$ has exactly two distinct children, unless $x_r \in \{-7, \overline{\sigma}+3 \}$. To do that we momentarily consider the isomorphic graph $G_{\theta_{k_{\sigma}}}^{p^n}$ and the vertex $\gamma_r = \chi^{-1}_{k_{\sigma}} (x_r)$ belonging to the level $r$ of the tree rooted in $\chi^{-1}_{k_{\sigma}} (x_0)$. The vertex $\gamma_r$ has exactly two distinct children, unless the discriminant of the quadratic equation $x^2 - \frac{\gamma_r}{k_{\sigma}} x + 1 =0$ is zero, namely $\gamma_r^2 - 4 k_{\sigma}^2 = 0$. This latter happens if and only if $\gamma_r = \pm 2 k_{\sigma}$. We prove now that
\begin{eqnarray*}
\chi_{k_{\sigma}} (2 k_{\sigma}) & = & - 7;\\
\chi_{k_{\sigma}} (- 2 k_{\sigma}) & = & \overline{\sigma} + 3. 
\end{eqnarray*}
Actually, both the assertions can be proved by explicit computation. Let us consider the first assertion.
\begin{eqnarray*}
\chi_k (2 k_{\sigma}) & = & \frac{3 k_{\sigma}+\frac{5}{2}}{\left(\frac{\frac{1}{2}-k_{\sigma}}{4 k_{\sigma}+1}\right)} = \frac{7 k_{\sigma} - \frac{7}{2}}{\frac{1}{2}-k_{\sigma}} = -7.
\end{eqnarray*}
Finally consider the second assertion. 
We have that
\begin{eqnarray*}
\chi_k (-2 k_{\sigma}) & = & \frac{- k_{\sigma}+\frac{5}{2}}{\left(\frac{\frac{1}{2} + 3 k_{\sigma}}{4 k_{\sigma}+1}\right)} = \frac{11 k_{\sigma} + \frac{9}{2}}{\frac{1}{2}+ 3 k_{\sigma}} = - \sigma + 4 = - \sigma^2 + 2 = \overline{\sigma} + 3
\end{eqnarray*}
by Lemma \ref{case3_lem_3}. By the same lemma,
\begin{equation*}
\eta_{k_{\sigma}} (-7) = \sigma +3 \quad \text{and} \quad \eta_{k_{\sigma}} (\sigma + 3) = \sigma + 3,
\end{equation*}
while 
\begin{equation*}
\eta_{k_{\sigma}} (\overline{\sigma} + 3) = \infty \quad \text{and} \quad \eta_{k_{\sigma}} (\infty) = \infty.
\end{equation*}
Hence, the thesis follows.
\end{proof}

\subsection{Structure of the graphs $G_{\theta_{k_{\sigma}}}^{p^n}$}
As stated in the first part of the current section, the graphs $G_{\theta_{k_{\sigma}}}^{p^n}$ are isomorphic to the graphs $G_{\eta_{k_{\sigma}}}^{p^n}$. Hence we can study the former relying upon the structure of the latter ones.
\begin{example}
Let $p=53$. Then, the two roots in $\F_p$ of the equation
\begin{equation*}
x^2 + \frac{1}{2} x + \frac{1}{2} = 0
\end{equation*}
are 
\begin{equation*}
k_{\omega} = 7 \text{ and } k_{\overline{\omega}} = 19,
\end{equation*}
being
\begin{equation*}
\omega = 15 \text{ and } \overline{\omega} = 39.
\end{equation*}
In this example we aim at describing thoroughly the structure of the graphs $G_{\theta_7}^{53}$ and $G_{\theta_{19}}^{53}$. At first we notice that $|E(\F_{53})| = 64$ and that the representation of the Frobenius endomorphism $\pi_{53}$ as an element of $R$ is
\begin{equation*}
\pi_{53} = - 7 + 4 \alpha,
\end{equation*}
being $\alpha = \frac{1+\sqrt{-7}}{2}$.
\end{example}

We focus now on the graph $G_{\theta_7}^{53} \cong G_{\eta_7}^{53}$. We notice that $\alpha \equiv 15 \pmod{\pi_{53}}$. Therefore we define $\rho_0 = \alpha$ and study the iterations of the map $\eta_7$ on $A_1$ by means of the iterations of $[\rho_0]$ on $S \cong R/(\pi_{53}-1) R$. We have that
\begin{equation*}
S = R/ \rho_0^4 R \times R / \rho_1^2 R,
\end{equation*}
where $\rho_1 = \overline{\rho}_0$. Since $e_0 = 4$, the depth of the trees rooted in elements of $A_1$ is $4$. The $\eta_7$-periodic elements in $A_1$ are $x$-coordinates of the points $P = (0, P_1)$ in $S$. In $R / {\rho}_1^2 R$ there are $2$ points of additive order $4$, which give rise to a cycle of length $1$, there is $1$ point of additive order $2$, which gives rise to a cycle of length $1$, and one more point of order $1$, which gives rise to one more cycle of length $1$. 

We concentrate now on the iterations of $\eta_7$ on $B_1$. To do that we consider the iterations of $[\rho_0]$ on  $S \cong R/(\pi_{53}+1) R$. We have that
\begin{equation*}
S = \prod_{i=0}^2 R / \rho_i R,
\end{equation*}
where
\begin{itemize}
\item $\rho_1 = \overline{\alpha}$;
\item $\rho_2 = -3 + 2 \alpha$ (and $N(\rho_2) = 11$).
\end{itemize}
Since $e_0 = 1$, the depth of the trees rooted in the elements of $B_1$ is $1$. The $\eta_7$-periodic elements in $B_1$ are $x$-coordinates of the points $P = (0, P_1, P_2)$ in $S$ such that, if $P_1$ has additive order $2$ in $R / {\rho}_1 R$ (namely $P_1 \not = 0$), then $P_2 \not = 0$ by Theorem \ref{case3_thm_1_B}. In $S$ there are exactly $10$ points $P$ where $P_1 \not = 0$ and $P_2 \not = 0$. The smallest among the positive integers $s$ such that $\rho_1$ divides either $\rho_0^s + 1 $ or $\rho_0^s-1$ is $1$. Moreover, the smallest among the positive integers $s$ such that $\rho_2$ divides either $\rho_0^s + 1 $ or $\rho_0^s-1$ is $5$. We can conclude that the $10$ points we are considering give rise to $1$ cycle of length $5$. The remaining $\eta_7$-periodic elements are $x$-coordinates of the $10$ points $P=(0,0,P_2) \in S$, where $P_2 \not = 0$. Such points give rise to another cycle of length $5$. 

The graph $G_{\theta_7}^{53}$, isomorphic to $G_{\eta_7}^{53}$, is below represented. We notice that the vertex labels, different from $\infty$ and `0' (the zero in $\F_{53}$), refer to the exponents of the powers $\gamma^i$ for $0 \leq i \leq 51$, being $\gamma$ a root of the Conway polynomial $x-2 \in \F_{53} [x]$. 

The following $3$ connected components are due to the elements of $A_1$.

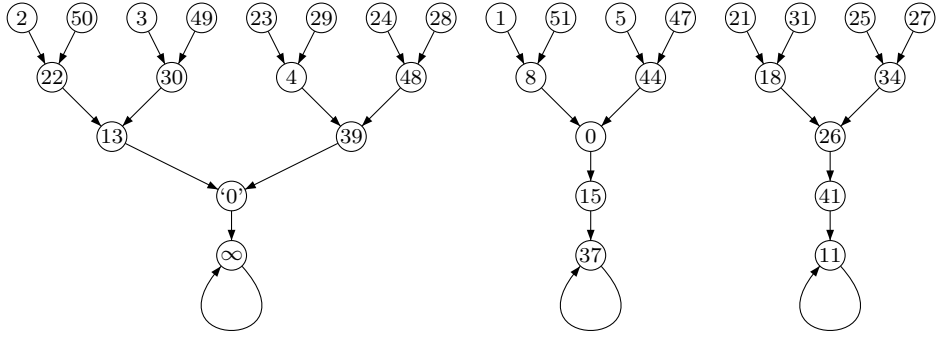
\begin{center}
    \unitlength=2.8pt
    \begin{picture}(125, 45)(5,-5)
    \gasset{Nw=4,Nh=4,Nmr=3,curvedepth=0}
    \thinlines
   \footnotesize
    
    \node(B1)(36,8){$\infty$}
    
    \node(C1)(36,16){`0'}
    
    \node(D1)(20,24){13}
    \node(D2)(52,24){39}
              
    \node(E1)(12,32){22}
   	\node(E2)(28,32){30}
    \node(E3)(44,32){4}
    \node(E4)(60,32){48}

	\node(F1)(8,40){2}
    \node(F2)(16,40){50}
    \node(F3)(24,40){3}          
    \node(F4)(32,40){49}
    \node(F5)(40,40){23}
    \node(F6)(48,40){29}
    \node(F7)(56,40){24}
    \node(F8)(64,40){28}
    
    \node(BB1)(84,8){37}
    \node(BB2)(116,8){11}
    
    \node(CC1)(84,16){15}
    \node(CC2)(116,16){41}
    
    \node(DD1)(84,24){0}
    \node(DD2)(116,24){26}
              
    \node(EE1)(76,32){8}
   	\node(EE2)(92,32){44}
    \node(EE3)(108,32){18}
    \node(EE4)(124,32){34}

	\node(FF1)(72,40){1}
    \node(FF2)(80,40){51}
    \node(FF3)(88,40){5}          
    \node(FF4)(96,40){47}
    \node(FF5)(104,40){21}
    \node(FF6)(112,40){31}
    \node(FF7)(120,40){25}
    \node(FF8)(128,40){27}

	\drawedge(F1,E1){}	
    \drawedge(F2,E1){}
    \drawedge(F3,E2){}
    \drawedge(F4,E2){}
    \drawedge(F5,E3){}
    \drawedge(F6,E3){}
    \drawedge(F7,E4){}
    \drawedge(F8,E4){}
    \drawedge(E1,D1){}
    \drawedge(E2,D1){}
    \drawedge(E3,D2){}
    \drawedge(E4,D2){}
    \drawedge(D1,C1){}
    \drawedge(D2,C1){}
    \drawedge(C1,B1){}
    \drawloop[loopangle=-90](B1){} 
    
    	\drawedge(FF1,EE1){}	
    \drawedge(FF2,EE1){}
    \drawedge(FF3,EE2){}
    \drawedge(FF4,EE2){}
    \drawedge(FF5,EE3){}
    \drawedge(FF6,EE3){}
    \drawedge(FF7,EE4){}
    \drawedge(FF8,EE4){}
    \drawedge(EE1,DD1){}
    \drawedge(EE2,DD1){}
    \drawedge(EE3,DD2){}
    \drawedge(EE4,DD2){}
    \drawedge(DD1,CC1){}
    \drawedge(DD2,CC2){}
    \drawedge(CC1,BB1){}
    \drawedge(CC2,BB2){}
    \drawloop[loopangle=-90](BB1){} 
    \drawloop[loopangle=-90](BB2){} 
    \end{picture}
  \end{center}
  
The following $2$ connected components are due to the elements of $B_1$.

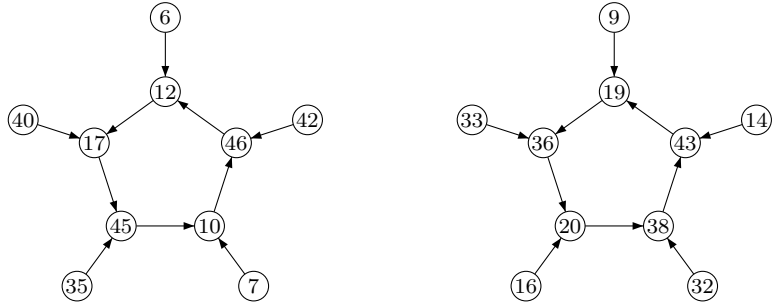
\begin{center}
\begin{picture}(60, 45)(-5,-20)
	\unitlength=2.8pt
    \gasset{Nw=4,Nh=4,Nmr=3,curvedepth=0}
    \thinlines
   \footnotesize
    \node(N1)(0,10){12}
    \node(N2)(9.5,3.1){46}
    \node(N3)(5.9,-8.1){10}
    \node(N4)(-5.9,-8.1){45}
    \node(N5)(-9.5,3.1){17}
    \node(N6)(0,20){6}
    \node(N7)(19,6.2){42}
    \node(N8)(11.8,-16.2){7}
    \node(N9)(-11.8,-16.2){35}
    \node(N10)(-19,6.2){40}
    \drawedge(N2,N1){}
    \drawedge(N3,N2){}
    \drawedge(N4,N3){}
    \drawedge(N5,N4){}
    \drawedge(N1,N5){}
    \drawedge(N6,N1){}
    \drawedge(N7,N2){}
    \drawedge(N8,N3){}
    \drawedge(N9,N4){}
    \drawedge(N10,N5){}
    
    \node(N11)(60,10){19}
    \node(N12)(69.5,3.1){43}
    \node(N13)(65.9,-8.1){38}
    \node(N14)(54.1,-8.1){20}
    \node(N15)(50.5,3.1){36}
    \node(N16)(60,20){9}
    \node(N17)(79,6.2){14}
    \node(N18)(71.8,-16.2){32}
    \node(N19)(48.2,-16.2){16}
    \node(N20)(41,6.2){33}
    \drawedge(N12,N11){}
    \drawedge(N13,N12){}
    \drawedge(N14,N13){}
    \drawedge(N15,N14){}
    \drawedge(N11,N15){}
    \drawedge(N16,N11){}
    \drawedge(N17,N12){}
    \drawedge(N18,N13){}
    \drawedge(N19,N14){}
    \drawedge(N20,N15){}
\end{picture}
\end{center}
  
We focus now on the graph $G_{\theta_{19}}^{53} \cong G_{\eta_{19}}^{53}$. Since $\overline{\alpha} \equiv 39 \pmod{\pi_{53}}$, we define $\rho_0 = \overline{\alpha}$ and study the iterations of the map $\eta_{19}$ on $A_1$ by means of the iterations of $[\rho_0]$ on $S \cong R/(\pi_{53}-1) R$. We have that
\begin{equation*}
S = R/ \rho_0^2 R \times R / \rho_1^4 R,
\end{equation*}
where $\rho_1 = \overline{\rho}_0$. Since $e_0 = 2$, the depth of the trees rooted in the elements of $A_1$ is $2$. The $\eta_{19}$-periodic elements in $A_1$ are $x$-coordinates of the points $P = (0, P_1)$ in $S$. The additive order of $P_1$ (and $P$) in $R / {\rho}_1^4 R$ (in $S$) can be $16, 8, 4, 2$, or $1$. The $8$ points $P$ of additive order $16$ give rise to $1$ cycle of length $4$, since $4$ is the smallest among the positive integers $s$ such that $\rho_1$ divides either $\rho_0^s+1$ or $\rho_s-1$. In a similar way we get that in $G_{\eta_{19}}^{53}$ there is $1$ cycle of length $2$ due to the $4$ points of additive order $8$ and $3$ cycles of length $1$ each due respectively to the $2$ points of additive order $4$, to the only point of order $2$ and to the only point of order $1$ (namely the point $(0,0)$).  

We concentrate now on the iterations of $\eta_{19}$ on $B_1$. To do that we consider the iterations of $[\rho_0]$ on  $S \cong R/(\pi_{53}+1) R$. In this case we have that
\begin{equation*}
S = \prod_{i=0}^2 R / \rho_i R,
\end{equation*}
where
\begin{itemize}
\item $\rho_1 = \alpha$;
\item $\rho_2 = -3 + 2 \alpha$ (and $N(\rho_2) = 11$).
\end{itemize}
The dynamics of $[\rho_0]$ on $B_1$ can be described employing the same arguments used for the description of the dynamics of $[\rho_0]$ over $B_1$ in the graph $G_{\eta_{7}}^{53}$ (actually, the connected components formed by the elements of the current set $B_1$ are isomorphic to the connected components corresponding to the former set $B_1$). 

The graph $G_{\theta_{19}}^{53}$, isomorphic to $G_{\eta_{19}}^{53}$, is represented below. Firstly we represent the $5$ connected components due to the elements of $A_1$.
\begin{center}
	\begin{picture}(60, 65)(-5,-35)
    	\unitlength=2.8pt
    \gasset{Nw=4,Nh=4,Nmr=3,curvedepth=0}
    \thinlines
   \footnotesize
    \node(A1)(10,0){17}
    \node(A2)(0,10){16}
    \node(A3)(-10,0){43}          
    \node(A4)(0,-10){42}
     
    \drawedge(A1,A2){}
    \drawedge(A2,A3){}
    \drawedge(A3,A4){}
    \drawedge(A4,A1){}
     
    \node(B1)(20,0){10}
    \node(B2)(0,20){35}
    \node(B3)(-20,0){36}          
    \node(B4)(0,-20){9}
    
    \drawedge(B1,A1){}
    \drawedge(B2,A2){}
    \drawedge(B3,A3){}
    \drawedge(B4,A4){}
    
    \node(C1)(30,10){4}
    \node(C2)(10,30){46}
    \node(C3)(-10,30){6}          
    \node(C4)(-30,10){22}
    \node(C5)(-30,-10){30}
    \node(C6)(-10,-30){20}
    \node(C7)(10,-30){32}          
    \node(C8)(30,-10){48}
    
    \drawedge(C1,B1){}
    \drawedge(C8,B1){}
    \drawedge(C2,B2){}
    \drawedge(C3,B2){}
    \drawedge(C4,B3){}
    \drawedge(C5,B3){}
    \drawedge(C6,B4){}
    \drawedge(C7,B4){}

    \node(AA1)(60,10){33}   
    \node(AA2)(60,-10){7}
     
    \gasset{curvedepth=2}
     
    \drawedge(AA1,AA2){}
    \drawedge(AA2,AA1){}
 
 	\gasset{curvedepth=0}
     
    \node(BB1)(60,20){45}
    \node(BB2)(60,-20){19}          
    
    \drawedge(BB1,AA1){}
    \drawedge(BB2,AA2){}

    \node(CC1)(70,30){2}
    \node(CC2)(50,30){50}          
    \node(CC3)(50,-30){24}
    \node(CC4)(70,-30){28}          
        
    \drawedge(CC1,BB1){}
    \drawedge(CC2,BB1){}
    \drawedge(CC3,BB2){}
    \drawedge(CC4,BB2){}

    \end{picture}
\end{center}

\begin{center}
	\begin{picture}(60, 40)(5,-10)
    	\unitlength=2.8pt
    \gasset{Nw=4,Nh=4,Nmr=3,curvedepth=0}
    \thinlines
   \footnotesize
    \node(A1)(0,0){14}
    \node(A2)(0,10){38}
    \node(A3)(0,20){0}

    \drawedge(A2,A1){}
    \drawedge(A3,A2){}

    \node(B1)(20,0){40}
    \node(B2)(20,10){12}
    \node(B3)(20,20){26}          

 	\drawedge(B2,B1){}
    \drawedge(B3,B2){}
    
    \node(C1)(70,0){$\infty$}
    \node(C2)(70,10){`0'}
    \node(C3)(60,20){13}
    \node(C4)(80,20){39}           

    \drawedge(C2,C1){}
    \drawedge(C3,C2){}
    \drawedge(C4,C2){}

	\drawloop[loopangle=-90](A1){} 
 	\drawloop[loopangle=-90](B1){} 
    \drawloop[loopangle=-90](C1){}

    \end{picture}
\end{center}

Now we represent the $2$ connected components due to the elements of $B_1$.

\begin{center}
\begin{picture}(60, 45)(-5,-20)
	\unitlength=2.8pt
    \gasset{Nw=4,Nh=4,Nmr=3,curvedepth=0}
    \thinlines
   \footnotesize
    \node(N1)(0,10){1}
    \node(N2)(9.5,3.1){49}
    \node(N3)(5.9,-8.1){34}
    \node(N4)(-5.9,-8.1){41}
    \node(N5)(-9.5,3.1){31}
    \node(N6)(0,20){3}
    \node(N7)(19,6.2){18}
    \node(N8)(11.8,-16.2){11}
    \node(N9)(-11.8,-16.2){21}
    \node(N10)(-19,6.2){51}
    \drawedge(N2,N1){}
    \drawedge(N3,N2){}
    \drawedge(N4,N3){}
    \drawedge(N5,N4){}
    \drawedge(N1,N5){}
    \drawedge(N6,N1){}
    \drawedge(N7,N2){}
    \drawedge(N8,N3){}
    \drawedge(N9,N4){}
    \drawedge(N10,N5){}
    
    \node(N11)(60,10){5}
    \node(N12)(69.5,3.1){27}
    \node(N13)(65.9,-8.1){23}
    \node(N14)(54.1,-8.1){8}
    \node(N15)(50.5,3.1){15}
    \node(N16)(60,20){25}
    \node(N17)(79,6.2){29}
    \node(N18)(71.8,-16.2){44}
    \node(N19)(48.2,-16.2){37}
    \node(N20)(41,6.2){47}
    \drawedge(N12,N11){}
    \drawedge(N13,N12){}
    \drawedge(N14,N13){}
    \drawedge(N15,N14){}
    \drawedge(N11,N15){}
    \drawedge(N16,N11){}
    \drawedge(N17,N12){}
    \drawedge(N18,N13){}
    \drawedge(N19,N14){}
    \drawedge(N20,N15){}
\end{picture}
\end{center}

\bibliography{Refs}
\end{document}